\documentclass[11pt]{article}
\usepackage{amsfonts}
\usepackage{amssymb}
\usepackage{graphics}
\usepackage{amsmath,epsfig,psfrag}
\usepackage[english]{babel}
\usepackage[all]{xy}
\usepackage{amscd}
\usepackage{amsthm}
\usepackage{latexsym}
\usepackage{graphicx}
\usepackage{mathrsfs}
\usepackage{multirow}
\usepackage{listings}
\usepackage{courier}
\usepackage{chngpage}
\usepackage{makecell}

\def\Z{{\Bbb Z}}

\def\a{{\alpha}}

\newtheorem{theorem}{Theorem}[section]

\newtheorem{lemma}[theorem]{Lemma}

\theoremstyle{definition}
\newtheorem{example}[theorem]{Example}

\theoremstyle{remark}

\begin{document}
\title{Bordered surfaces in the 3-sphere with maximum symmetry}
\author{Chao Wang, Shicheng Wang, Yimu Zhang and Bruno Zimmermann}

\date{}
\maketitle

\begin{abstract}

We consider orientation-preserving actions of finite groups $G$
on pairs $(S^3, \Sigma)$, where $\Sigma$ denotes a compact connected surface
embedded in $S^3$. In a previous paper, we considered the case of closed, necessarily orientable surfaces, determined for each genus $g>1$
the maximum order of such a $G$ for all embeddings of a
surface of genus $g$, and classified the corresponding
embeddings.

\smallskip

In the present paper we obtain analogous results for the case of
bordered surfaces $\Sigma$ (i.e. with non-empty boundary, orientable or
not). Now the genus $g$ gets replaced by the algebraic genus $\alpha$ of
$\Sigma$ (the rank of its free fundamental group); for each $\alpha > 1$
we determine the maximum order $m_\alpha$ of an action of $G$,
classify the topological types of the corresponding surfaces (topological genus, number of boundary components, orientability) and their embeddings into
$S^3$. For example, the maximal possibility $12(\alpha - 1)$
is obtained for the finitely many values $\alpha = 2, 3, 4, 5, 9, 11,
25, 97, 121$ and $241$.

\end{abstract}

\tableofcontents

\section{Introduction}
We study smooth, faithful actions of finite groups $G$ on pairs
$(S^3, \Sigma)$ where $\Sigma$ denotes a compact, connected, bordered surface
with an embedding $e: \Sigma\to S^3$ (so $G$ is a finite group of diffeomorphisms of a pair $(S^3, \Sigma)$). We also say that such a $G$-action
on $\Sigma$ is \textit{extendable} (w.r.t. $e$).

Let $\Sigma_{g, b}$ denote the orientable compact surface of (topological) genus $g$ with $b$ boundary components, writing also $\Sigma_g$ instead of $\Sigma_{g,0}$; for $g>0$, let $\Sigma^-_{g, b}$ denote the non-orientable compact surface of genus $g$ with $b$ boundary components. $\Sigma^-_{g, b}$ is obtained from the connected sum of $g$ real projective planes by creating  $b$ boundary components (by deleting the interiors of $b$ disjoint embedded disks), and it is well-known that each compact surface is either $\Sigma_{g, b}$ or $\Sigma^-_{g, b}$.

For $b>0$, $\Sigma_{g, b}$ and $\Sigma^-_{g, b}$ are bordered surfaces, and we use $\alpha(\Sigma)$ to denote their algebraic genus equal to the rank of the free fundamental group $\pi_1(\Sigma)$; this is also the genus of a regular neighbourhood of $\Sigma$ in $S^3$ which is a 3-dimensional handlebody. We have $\alpha(\Sigma_{g, b})=2g-1+b$ and $\alpha(\Sigma^-_{g, b})=g-1+b$.
We will always assume that $\alpha>1$ in the present paper.

We will consider only orientation-preserving
finite group actions on $S^3$; then, referring to the recent
geometrization of finite group actions on $S^3$ (see \cite{BMP} for the case of non-free actions and \cite{Pe} for the general case), we can restrict to orthogonal actions of finite groups on $S^3$, i.e. to finite
subgroups $G$ of the orthogonal group $SO(4)$.

Let $m_\alpha$ denote the maximum order of such a group $G$ acting on a
pair $(S^3, \Sigma)$, for all embeddings of bordered surfaces $\Sigma$
of a fixed algebraic genus $\alpha$ into $S^3$. In the present paper we will determine
$m_\alpha$ and classify all surfaces $\Sigma$ which realize the
maximum order $m_\alpha$.

A similar question for the pair $(S^3, \Sigma_g)$, where $\Sigma_g$
is the closed orientable surface of genus $g$, was studied in
\cite{WWZZ}. The corresponding maximum order $OE_g$ of finite
groups acting on $(S^3, \Sigma_g)$ for all possible embeddings
$\Sigma_g\subset S^3$ was obtained in that paper.

Let $V_g$ denote the handlebody of genus $g$. Each bordered surface
$\Sigma\subset S^3$ of algebraic genus $
\alpha$ has a regular neighborhood which is homeomorphic to $V_\alpha$. We note that, similar as for handlebodies, the maximal possibilities for the orders of groups of homeomorphisms of compact bordered surfaces of algebraic genus $\a$ are $12(\a-1)$, $8(\a-1)$, $20(\a-1)/3$, $6(\a-1)$, ... (see section 3 of \cite{MZ}), and these are exactly the values occurring in the next theorem. A classification of all finite group actions on compact bordered surfaces up to algebraic genus $101$ is given in \cite{BCC}, and the lists in that paper may be compared with the list in the next theorem.
Concerning other papers considering symmetries of surfaces immersed in 3-space, see \cite{CC}, \cite{CH} and \cite{LT}.

Our main result is:
\newpage
\renewcommand\arraystretch{1.3}
\begin{theorem}\label{mg}
For each $\a>1$,  $m_\a$ and the surfaces realizing $m_\a$ are listed below.
\begin{center}
\begin{adjustwidth}{-3.2em}{0em}
\begin{tabular}{|c|c|c|}
\hline  $\a$ & $m_\a$ & $\Sigma$\\\hline\hline
$2$ & $12(\a-1)=12$ & $\Sigma_{0,3}$, $\Sigma_{1,1}$\\\hline
$3$ & $12(\a-1)=24$ & $\Sigma_{0,4}$, $\Sigma^-_{1,3}$\\\hline
$4$ & $12(\a-1)=36$ & $\Sigma_{1,3}$\\\hline
$5$ & $12(\a-1)=48$ & $\Sigma_{0,6}$, $\Sigma_{1,4}$\\\hline
$9$ & $12(\a-1)=96$ & $\Sigma_{2,6}$, $\Sigma_{3,4}$\\\hline
$11$ & $12(\a-1)=120$ & $\Sigma_{0,12}$, $\Sigma^-_{6,6}$\\\hline
$25$ & $12(\a-1)=288$ & $\Sigma_{7,12}$, $\Sigma_{10,6}$\\\hline
$97$ & $12(\a-1)=1152$ & $\Sigma_{37,24}$\\\hline
$121$ & $12(\a-1)=1440$ & $\Sigma_{43,36}$, $\Sigma_{55,12}$\\\hline
$241$ & $12(\a-1)=2880$ & $\Sigma_{73,96}$, $\Sigma_{97,48}$, $\Sigma^-_{206,36}$\\\hline
\hline
$7$ & $8(\a-1)=48$ & $\Sigma_{0,8}$, $\Sigma^-_{4,4}$\\\hline
$49$ & $8(\a-1)=384$ & $\Sigma_{17,16}$, $\Sigma_{21,8}$\\\hline\hline
$16$ & $\frac{20}{3}(\a-1)=100$ & $\Sigma_{6,5}$\\\hline
$19$ & $\frac{20}{3}(\a-1)=120$ & $\Sigma_{0,20}$, $\Sigma^-_{14,6}$\\\hline
$361$ & $\frac{20}{3}(\a-1)=2400$ & $\Sigma_{131,100}$, $\Sigma_{151,60}$, $\Sigma_{171,20}$\\\hline\hline
$21$ & $6(\a-1)=120$ & $\Sigma_{5,12}$\\\hline
$481$ & $6(\a-1)=2880$ & $\Sigma_{205,72}$, $\Sigma_{193,96}$\\\hline\hline
$41$ & $\frac{24}{5}(\a-1)=192$ & $\Sigma^-_{30,12}$\\\hline\hline
$1681$ & $\frac{30}{7}(\a-1)=7200$ & $\Sigma^-_{1562,120}$\\\hline\hline
$841$& $4(\sqrt{\a}+1)^2=3600$ & $\Sigma_{391,60}$, $\Sigma_{406,30}$\\\hline
$k^2$, $k \neq 3,5,7,11,19,41$ & $4(\sqrt{\a}+1)^2$ & $\Sigma_{\frac{k(k-1)}{2},k+1}$\\\hline\hline
$29$& $4(\a+1)=120$ & $\Sigma_{0,30}$, $\Sigma_{9,12}$, $\Sigma_{14,2}$\\\hline
 \text{the remaining numbers}& $4(\a+1)$ & $\Sigma_{0,\a+1}$, $\Sigma_{\frac{\a}{2},1}$ \text{($\a$ even)},  $\Sigma_{\frac{\a-1}{2},2}$ \text{($\a$ odd)},\\\hline
\end{tabular}
\end{adjustwidth}
\end{center}
\end{theorem}

\bigskip\noindent\textbf{Acknowledgement}.
This work was supported by National Natural Science Foundation of China (Grant Nos.
11371034 and 11501239).

\section{The case of closed surfaces: 3-orbifolds and main results in \cite{WWZZ}}\label{orbifold}

For orbifold theory, see \cite{Th}, \cite{Du1} or \cite{BMP}. We
give a brief introduction here for later use.

All of the $n$-orbifolds that we consider have the form $M/H$. Here $M$ is an
orientable $n$-manifold and $H$ is a finite group acting faithfully
on $M$, preserving orientation. For each point $x\in M$, denote
its stable subgroup by $St(x)$, its image in $M/H$ by $x'$. If
$|St(x)|>1$, $x'$ is called a \textit{singular} point with \textit{index} $|St(x)|$,
otherwise it is called a \textit{regular} point. If we forget the singular
set we get the topological \textit{underlying space} $|M/H|$ of the orbifold.

We can also define covering spaces and the fundamental group of an
orbifold. There is an one-to-one correspondence between orbifold
covering spaces and conjugacy classes of subgroups of the
fundamental group, and regular covering spaces correspond to normal
subgroups. A Van-Kampen theorem is also valid, see \cite[Corollary 2.3]{BMP}. In
the following, automorphisms, covering spaces and fundamental groups
always refer to the orbifold setting.

We call $B^n/H$ (resp. $S^n/H,V_g/H)$ a \textit{discal (resp. spherical, handlebody)
orbifold}. Here $B^n$ (resp. $S^n$) denotes the $n$-dimensional ball (resp. sphere).
By classical results, $B^2/H$ is a disk, possibly with one singular
point; $B^3/H$ belongs to one of the five models in Figure
\ref{fig:1}, corresponding to the five classes of finite subgroups
of $SO(3)$. Here the labeled numbers denote indices of interior
points of the corresponding edges. $V_g/H$ can be obtained by
pasting finitely many $B^3/H$ along some $B^2/H$ in their
boundaries. It is easy to see that the singular set of a 3-orbifold $M/H$
is always a trivalent graph $\Theta$.

\begin{figure}[h]
\centerline{\scalebox{0.6}{\includegraphics{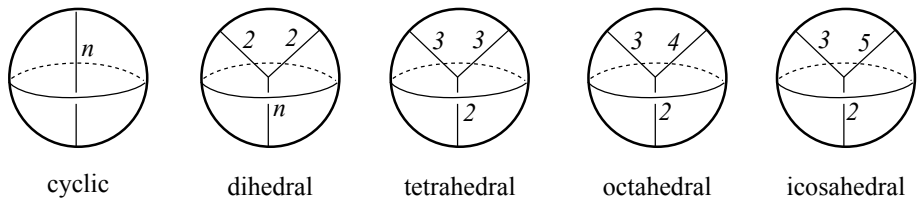}}} \caption{Five
models}\label{fig:1}
\end{figure}

Suppose $G$ acts on $(S^3, \Sigma_g)$, where $\Sigma_g$ is a closed surface. Call a 2-orbifold
$\mathcal{F}=\Sigma_g/G \subset \mathcal{O}=S^3/G$ {\it
allowable} if $|G|>4(g-1)$. A sequence of observations about
allowable 2-orbifolds were made in \cite{WWZZ} (Lemma 2.4, 2.7, 2.8,
2.9, 2.10), in particular: Suppose $\mathcal{F} \subset \mathcal{O}$
is allowable, then

(i) $|\mathcal{O}|=S^3$;

(ii) $\mathcal{F}\subset
\mathcal{O}$ is $\pi_1$-surjective;

(iii) $|\mathcal{F}|=S^2$ with
four singular points having one of the following types:
$(2,2,2,n)(n\ge 3)$, $(2,2,3,3)$, $(2,2,3,4)$, $(2,2,3,5)$;

(iv) $\mathcal{F}$ bounds a handlebody orbifold which is
a regular neighborhood of either an edge of the singular set or a
dashed arc, presented in (a) or (b) of Figure \ref{fig:2} (so a dashed arc does not belong to the singular set). Here
labels are omitted in (a), and more description of (b) will be given
later.

(i) allows us to consider only Dunbar's famous list in
\cite{Du1} of all spherical 3-orbifolds with underlying space $S^3$.
Searching for all possible 2-suborbifolds that satisfy the
conditions (ii), (iii) and (iv) by further analysis from
topological, combinatoric, numerical, and group theoretical aspects
leads to a list in Theorem 6.1 of \cite{WWZZ}, presented here as
Theorem \ref{classify}. We will first need
to explain the terminology in the statement of Theorem \ref{classify} and the notation in the accompanying tables.

\begin{figure}[h]
\centerline{\scalebox{0.5}{\includegraphics{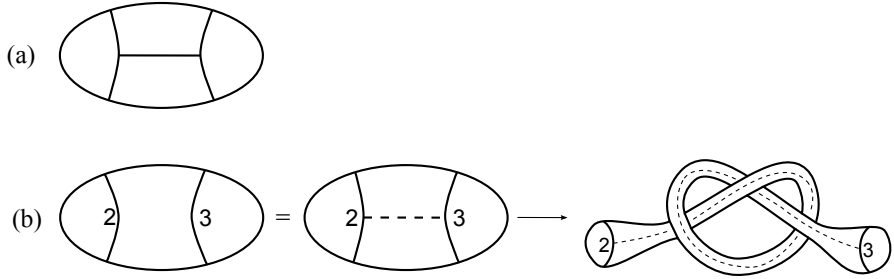}}}
\caption{Handlebody orbifolds}\label{fig:2}
\end{figure}

Since all the spherical 3-orbifolds we consider have underlying
space $S^3$, they are determined by their embedded labeled singular trivalent
graphs. From now on, a \emph{singular edge} always means an edge of
$\Theta$, the singular set of the orbifold;
singular edges with index $2$ are not labeled; and a \emph{dashed arc} is
always a regular arc with two ends at two edges of $\Theta$ with
indices 2 and 3 as in Figure \ref{fig:2}(b). An
edge/dashed arc is {\it allowable} if the boundary of its regular
neighborhood is an allowable 2-orbifold.

For each 3-orbifold $\mathcal{O}$ in the list of Theorem \ref{classify}, the order of
$\pi_1(\mathcal{O})$ is given first. Then singular edges/dashed
arcs are listed, which are marked by letters $a, b, c,...$ to denote
the boundaries of their regular neighborhoods. Then singular types
of the boundaries and genera of their pre-images in $S^3$ are given.
When the singular type is $(2, 2, 3, 3)$, there are two subtypes
denoted by I and II, corresponding to Figure \ref{fig:2}(a) and
Figure \ref{fig:2}(b) (exactly the dashed arc case).

We say that an orientable separating 2-suborbifold (2-subsurface)
$\mathcal{F}$ in an orientable 3-orbifold (3-manifold) $\mathcal{O}$
is \textit{unknotted} or \textit{knotted}, depending on whether it bounds handlebody
orbifolds (handlebodies) on both sides. A singular edge/dashed arc
is \textit{unknotted} or \textit{knotted}, depending on whether the boundary of its
regular neighborhood is unknotted or knotted.

If a marked singular edge/dashed arc is knotted, then it has a subscript `$k$'. If a marked dashed arc is unknotted, then there also
exists a knotted one (indeed infinitely many) and it has a subscript `$uk$'. Call two singular edges/dashed arcs \textit{equivalent},
if there is an orbifold automorphism sending one to the other, or
the boundaries of their regular neighborhoods as 2-orbifolds are
orbifold-isotopic.

The way to list orbifolds in Theorem \ref{classify} is influenced by the lists of \cite{Du1} and \cite{Du2}. The labels below the orbifold pictures come from \cite[Tables I, II and III]{WWZZ}. In picture 15E and picture 19, the letter $n$ refers to particular choices of parameters in infinite families.

\begin{theorem}\label{classify}
Up to equivalence, the following tables list all allowable singular
edges/dashed arcs except those of type II. In the type II case, if
there exists an allowable dashed arc in some $\mathcal{O}$, then
$\mathcal{O}$ and one such arc in it are listed, and this arc in the list will be
unknotted if there exists an unknotted one in $\mathcal{O}$.
\end{theorem}
\newpage
\begin{table}[h]
\caption{Fibred case: type is $(2, 2, 3, 3)$}\label{tab:fibred2233}
\end{table}
\begin{center}
\scalebox{0.4}{\includegraphics*[0pt,0pt][152pt,152pt]{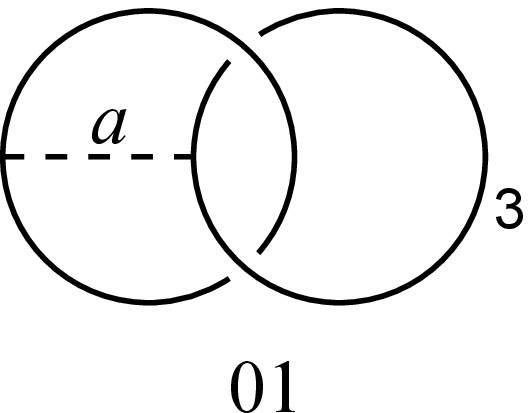}}
\raisebox{45pt} {\parbox[t]{102pt}{$|G|=6$\\$a_{uk}$: II, $g=2$}}
\scalebox{0.4}{\includegraphics*[0pt,0pt][152pt,152pt]{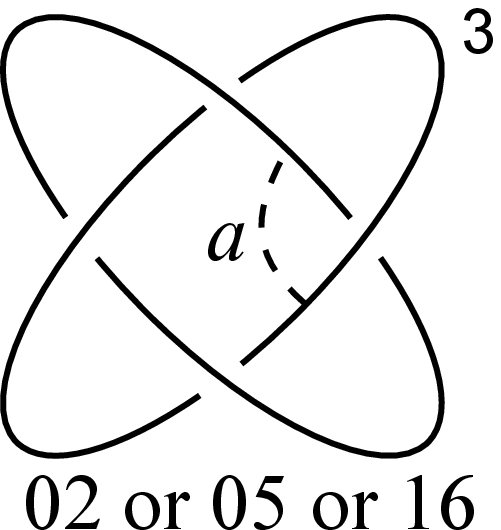}}
\raisebox{45pt} {\parbox[t]{102pt}{$|G|=18$\\$a_{uk}$: II, $g=4$}}
\end{center}

\begin{center}
\scalebox{0.4}{\includegraphics*[0pt,0pt][152pt,152pt]{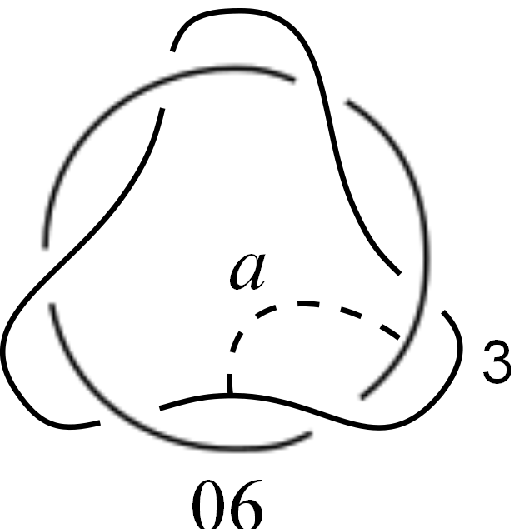}}
\raisebox{45pt} {\parbox[t]{102pt}{$|G|=48$\\$a_{uk}$: II, $g=9$}}
\scalebox{0.4}{\includegraphics*[0pt,0pt][152pt,152pt]{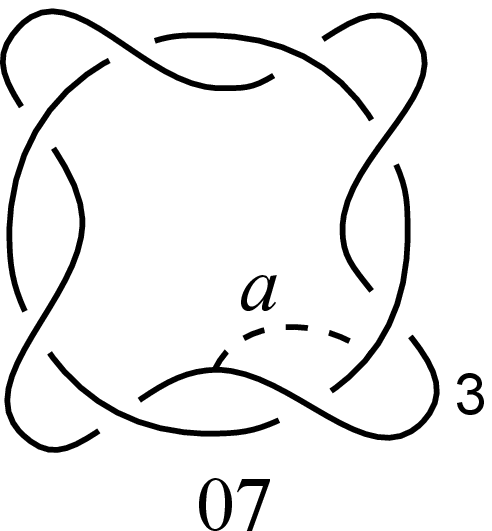}}
\raisebox{45pt} {\parbox[t]{102pt}{$|G|=144$\\$a_{uk}$: II, $g=25$}}
\end{center}

\begin{center}
\scalebox{0.4}{\includegraphics*[0pt,0pt][152pt,152pt]{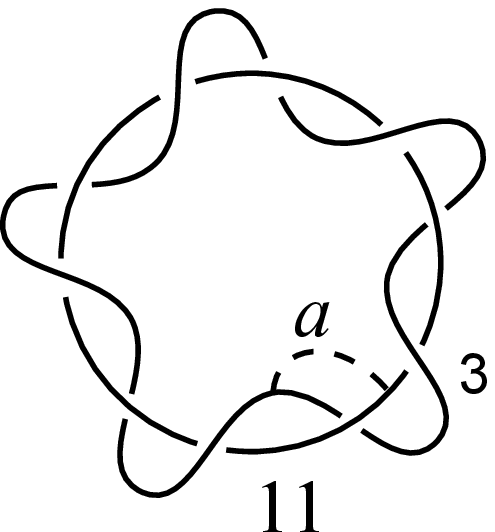}}
\raisebox{45pt} {\parbox[t]{102pt}{$|G|=720$\\$a_{uk}$: II,
$g=121$}}
\scalebox{0.4}{\includegraphics*[0pt,0pt][152pt,152pt]{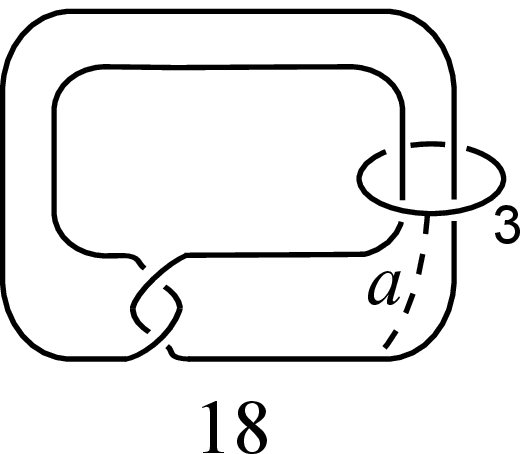}}
\raisebox{45pt} {\parbox[t]{102pt}{$|G|=144$\\$a_{uk}$: II, $g=25$}}
\end{center}

\begin{center}
\scalebox{0.4}{\includegraphics*[0pt,0pt][152pt,152pt]{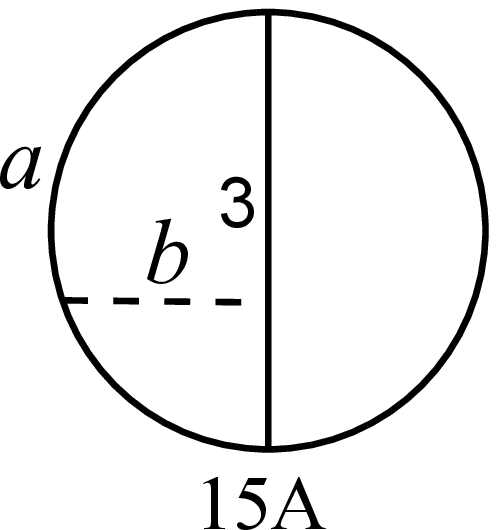}}
\raisebox{45pt} {\parbox[t]{102pt}{$|G|=6$\\$a$: I, $g=2$\\$b_{uk}$:
II, $g=2$}}
\scalebox{0.4}{\includegraphics*[0pt,0pt][152pt,152pt]{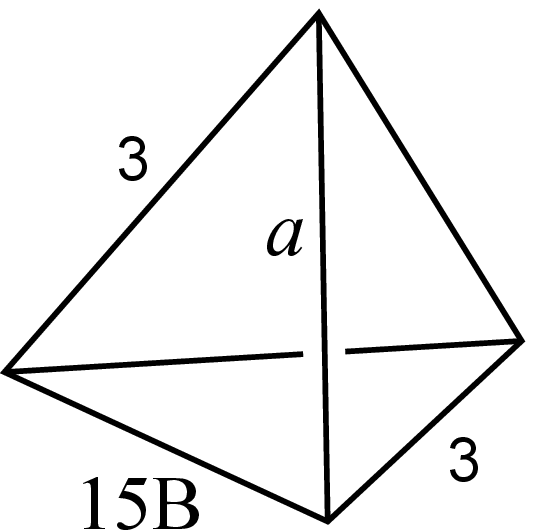}}
\raisebox{45pt} {\parbox[t]{102pt}{$|G|=18$\\$a$: I, $g=4$}}
\end{center}

\begin{center}
\scalebox{0.4}{\includegraphics*[0pt,0pt][152pt,152pt]{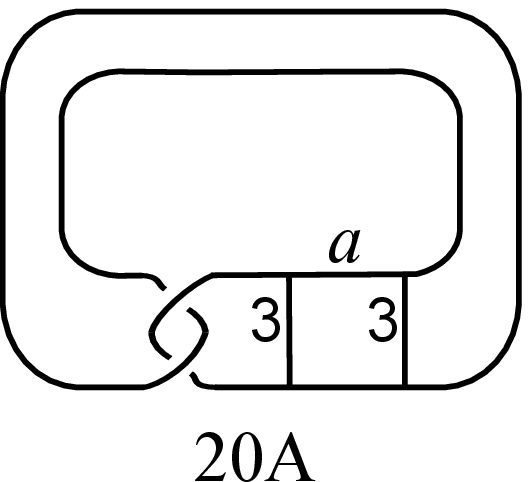}}
\raisebox{45pt} {\parbox[t]{102pt}{$|G|=144$\\$a$: I, $g=25$}}
\scalebox{0.4}{\includegraphics*[0pt,0pt][152pt,152pt]{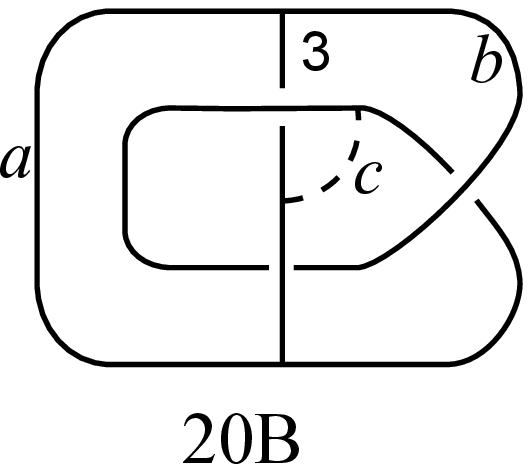}}
\raisebox{45pt} {\parbox[t]{102pt}{$|G|=48$\\$a$: I, $g=9$\\$b_k$:
I, $g=9$\\$c_{uk}$: II, $g=9$}}
\end{center}

\begin{center}
\scalebox{0.4}{\includegraphics*[0pt,0pt][152pt,152pt]{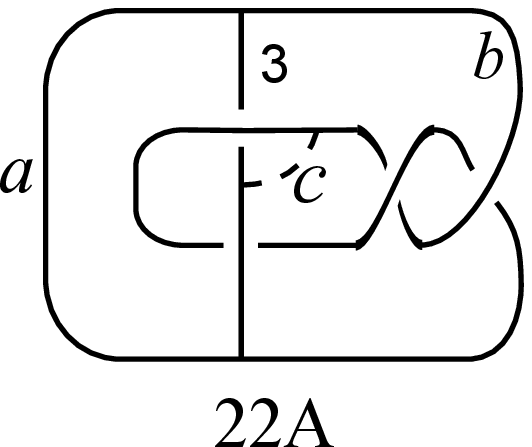}}
\raisebox{45pt} {\parbox[t]{102pt}{$|G|=720$\\$a$: I,
$g=121$\\$b_k$: I, $g=121$\\$c_{uk}$: II, $g=121$}} \hspace*{170pt}
\end{center}

\begin{table}[h]
\caption{Fibred case: type is not $(2, 2, 3,
3)$}\label{tab:fibnot2233}
\end{table}
\begin{center}
\scalebox{0.4}{\includegraphics*[0pt,0pt][152pt,152pt]{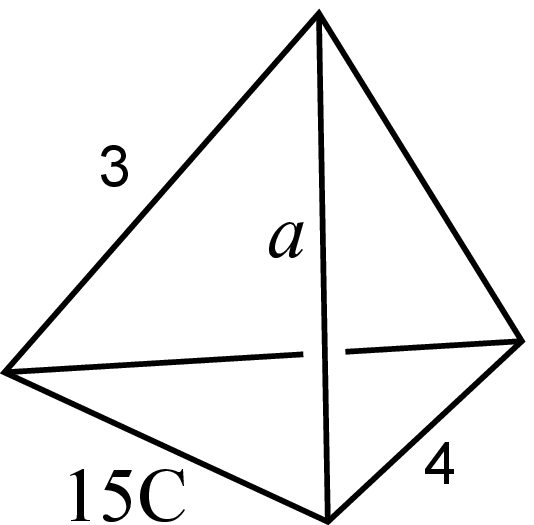}}
\raisebox{45pt} {\parbox[t]{102pt}{$|G|=24$\\$a$:(2,2,3,4), $g=6$}}
\scalebox{0.4}{\includegraphics*[0pt,0pt][152pt,152pt]{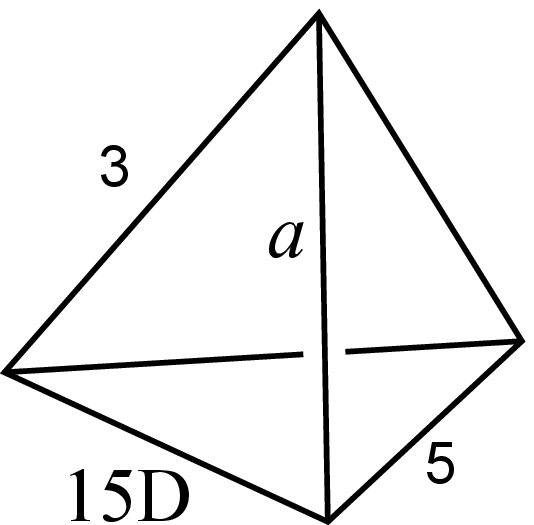}}
\raisebox{45pt} {\parbox[t]{102pt}{$|G|=30$\\$a$:(2,2,3,5), $g=8$}}
\end{center}

\begin{center}
\scalebox{0.4}{\includegraphics*[0pt,0pt][152pt,152pt]{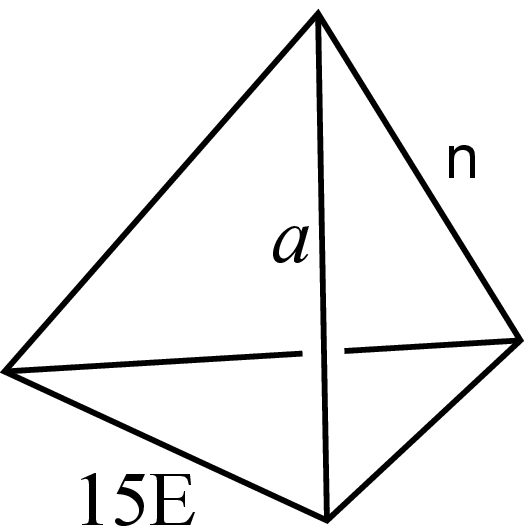}}
\raisebox{45pt}{\parbox[t]{102pt}{$|G|=4n$\\$a$:(2,2,2,n)\\$g=n-1$}}
\scalebox{0.4}{\includegraphics*[0pt,0pt][152pt,152pt]{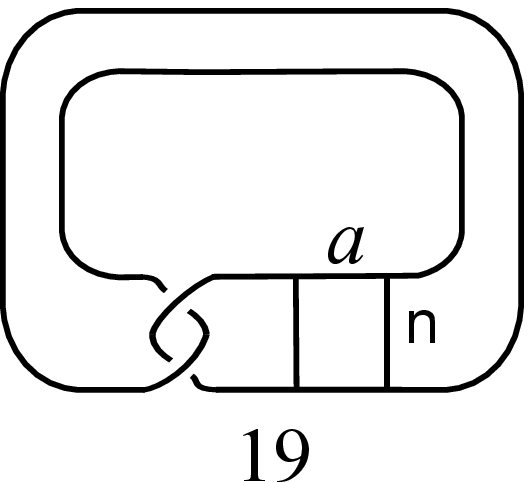}}
\raisebox{45pt}{\parbox[t]{102pt}{$|G|=4n^2$\\$a$:(2,2,2,n)\\$g=(n-1)^2$}}
\end{center}

\begin{center}
\scalebox{0.4}{\includegraphics*[0pt,0pt][152pt,152pt]{d20C.eps}}
\raisebox{45pt} {\parbox[t]{102pt}{$|G|=96$\\$a$:(2,2,2,3),
$g=9$\\$b_{k}$:(2,2,2,3), $g=9$}}
\scalebox{0.4}{\includegraphics*[0pt,0pt][152pt,152pt]{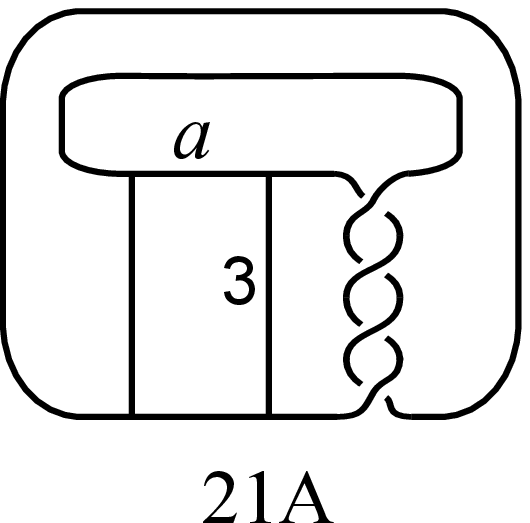}}
\raisebox{45pt} {\parbox[t]{102pt}{$|G|=288$\\$a$:(2,2,2,3),
$g=25$}}
\end{center}

\begin{center}
\scalebox{0.4}{\includegraphics*[0pt,0pt][152pt,152pt]{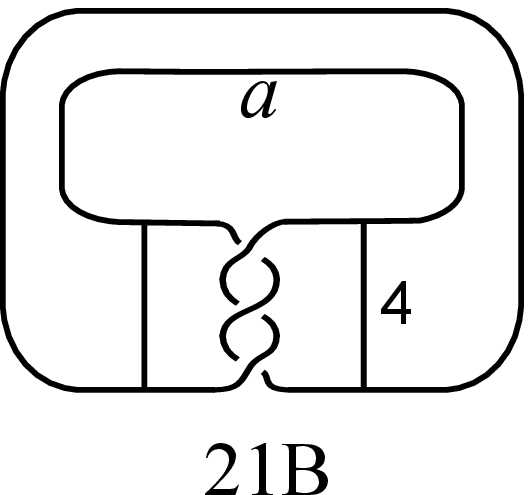}}
\raisebox{45pt} {\parbox[t]{102pt}{$|G|=384$\\$a$:(2,2,2,4),
$g=49$}}
\scalebox{0.4}{\includegraphics*[0pt,0pt][152pt,152pt]{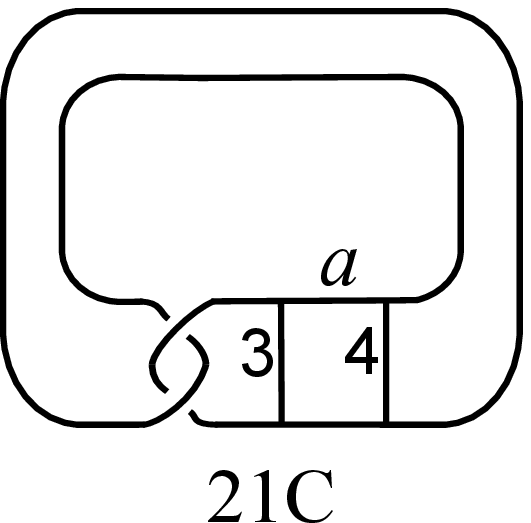}}
\raisebox{45pt} {\parbox[t]{102pt}{$|G|=576$\\$a$:(2,2,3,4),
$g=121$}}
\end{center}

\begin{center}
\scalebox{0.4}{\includegraphics*[0pt,0pt][152pt,152pt]{d22B.eps}}
\raisebox{45pt} {\parbox[t]{102pt}{$|G|=1440$\\$a$:(2,2,2,3),
$g=121$\\$b_{k}$:(2,2,2,3), $g=121$}}
\scalebox{0.4}{\includegraphics*[0pt,0pt][152pt,152pt]{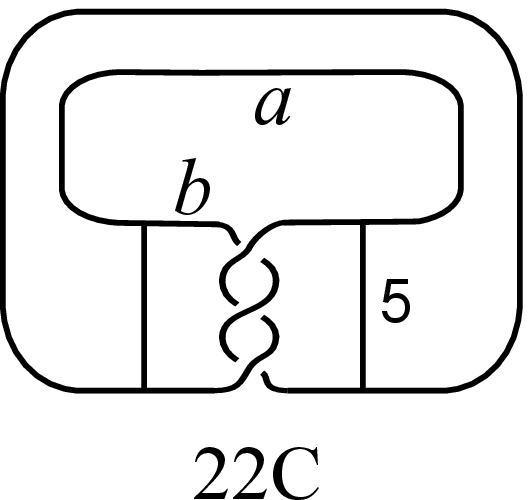}}
\raisebox{45pt} {\parbox[t]{102pt}{$|G|=2400$\\$a$:(2,2,2,5),
$g=361$\\$b_{k}$:(2,2,2,5), $g=361$}}
\end{center}

\begin{center}
\scalebox{0.4}{\includegraphics*[0pt,0pt][152pt,152pt]{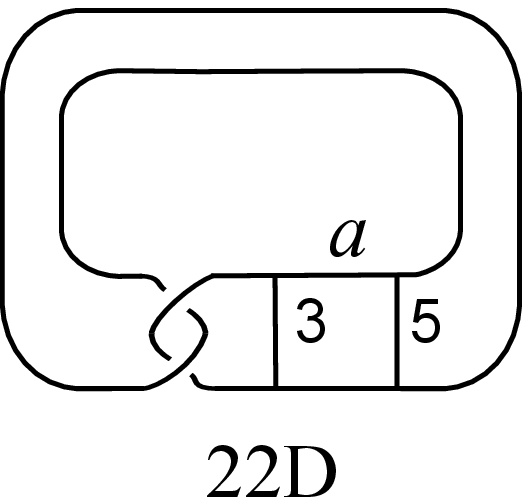}}
\raisebox{45pt} {\parbox[t]{102pt}{$|G|=3600$\\$a$:(2,2,3,5),
$g=841$}}\hspace*{170pt}
\end{center}

\begin{table}[h]
\caption{Non-fibred case}\label{tab:nonfibre}
\end{table}
\begin{center}
\scalebox{0.4}{\includegraphics*[0pt,0pt][152pt,152pt]{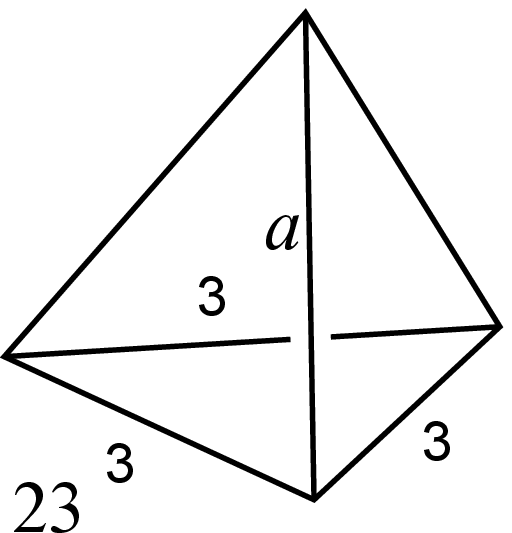}}
\raisebox{45pt} {\parbox[t]{102pt}{$|G|=96$\\$a$: I, $g=17$}}
\scalebox{0.4}{\includegraphics*[0pt,0pt][152pt,152pt]{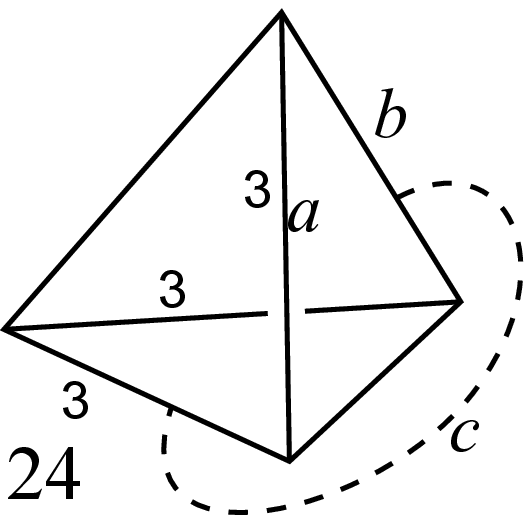}}
\raisebox{45pt} {\parbox[t]{102pt}{$|G|=60$\\$a$:(2,2,2,3),
$g=6$\\$b$: I, $g=11$\\$c_{k}$: II, $g=11$}}
\end{center}

\begin{center}
\scalebox{0.4}{\includegraphics*[0pt,0pt][152pt,152pt]{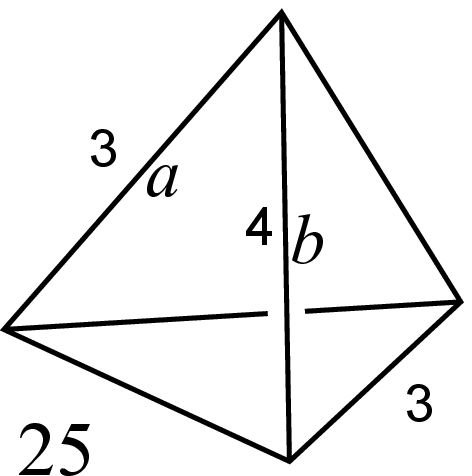}}
\raisebox{45pt} {\parbox[t]{102pt}{$|G|=576$\\$a$:(2,2,2,4),
$g=73$\\$b$: I, $g=97$}}
\scalebox{0.4}{\includegraphics*[0pt,0pt][152pt,152pt]{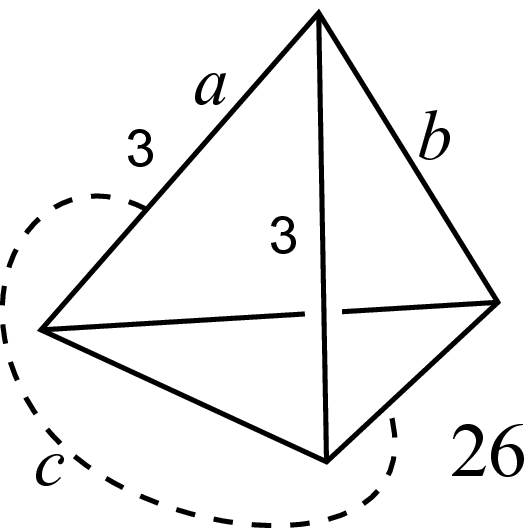}}
\raisebox{45pt} {\parbox[t]{102pt}{$|G|=24$\\$a$:(2,2,2,3),
$g=3$\\$b$: I, $g=5$\\$c_k$: II, $g=5$}}
\end{center}

\begin{center}
\scalebox{0.4}{\includegraphics*[0pt,0pt][152pt,152pt]{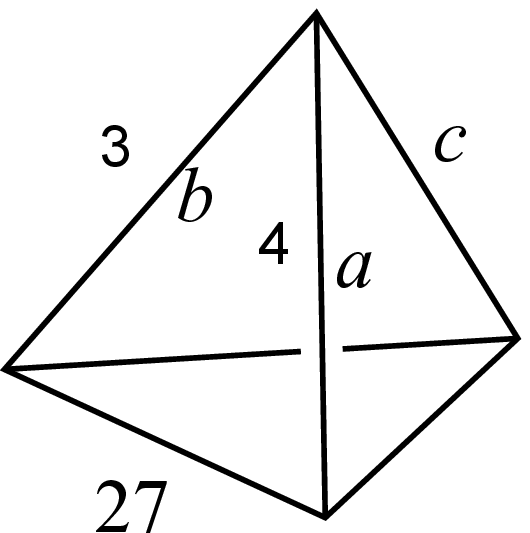}}
\raisebox{45pt} {\parbox[t]{102pt}{$|G|=48$\\$a$:(2,2,2,3),
$g=5$\\$b$:(2,2,2,4), $g=7$\\$c$:(2,2,3,4), $g=11$}}
\scalebox{0.4}{\includegraphics*[0pt,0pt][152pt,152pt]{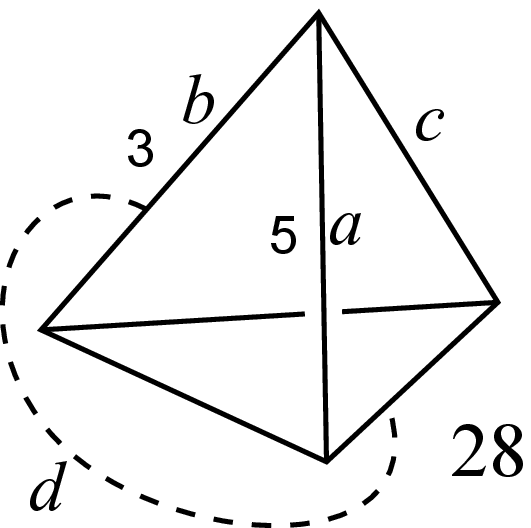}}
\raisebox{45pt} {\parbox[t]{102pt}{$|G|=120$\\$a$:(2,2,2,3),
$g=11$\\$b$:(2,2,2,5), $g=19$\\$c$:(2,2,3,5), $g=29$\\$d_{k}$: II,
$g=21$}}
\end{center}

\begin{center}
\scalebox{0.4}{\includegraphics*[0pt,0pt][152pt,152pt]{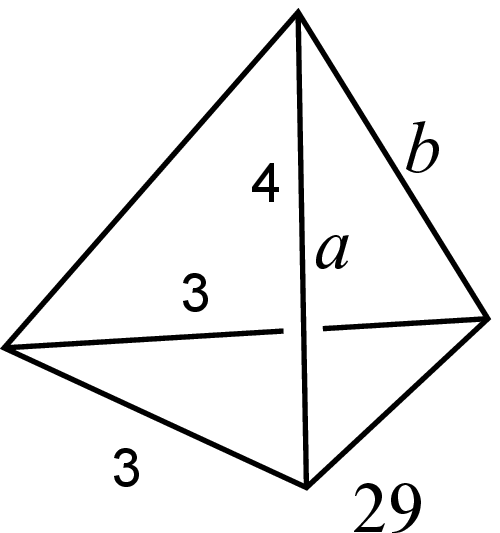}}
\raisebox{45pt} {\parbox[t]{102pt}{$|G|=192$\\$a$:(2,2,2,3),
$g=17$\\$b$:(2,2,3,4), $g=41$}}
\scalebox{0.4}{\includegraphics*[0pt,0pt][152pt,152pt]{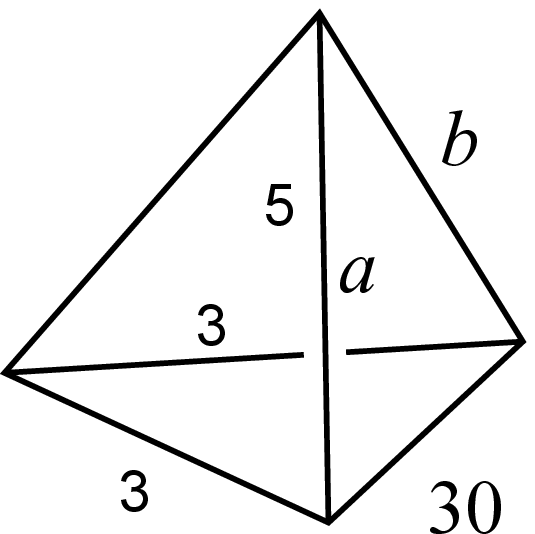}}
\raisebox{45pt} {\parbox[t]{102pt}{$|G|=7200$\\$a$:(2,2,2,3),
$g=601$\\$b$:(2,2,3,5), $g=1681$}}
\end{center}

\begin{center}
\scalebox{0.4}{\includegraphics*[0pt,0pt][152pt,152pt]{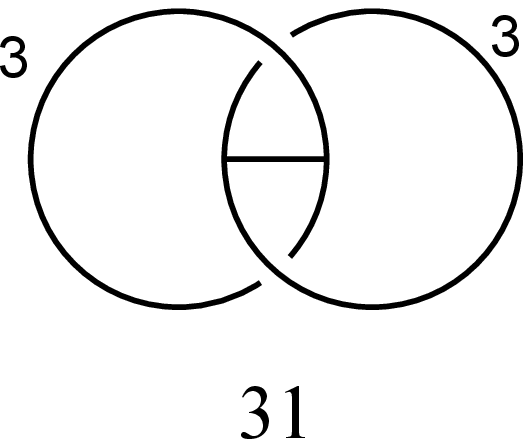}}
\raisebox{45pt} {\parbox[t]{102pt}{$|G|=288$\\No allowable\\
2-suborbifold}}
\scalebox{0.4}{\includegraphics*[0pt,0pt][152pt,152pt]{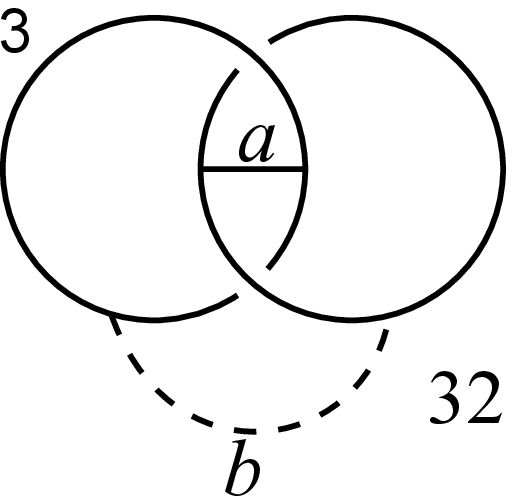}}
\raisebox{45pt} {\parbox[t]{102pt}{$|G|=24$\\$a$: I,
$g=5$\\$b_{uk}$: II, $g=5$}}
\end{center}

\begin{center}
\scalebox{0.4}{\includegraphics*[0pt,0pt][152pt,152pt]{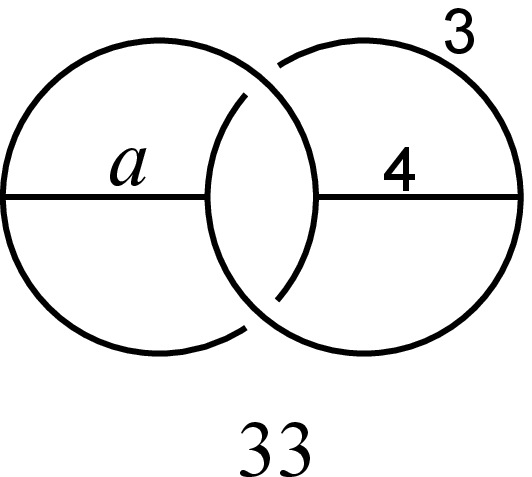}}
\raisebox{45pt} {\parbox[t]{102pt}{$|G|=1152$\\$a$:(2,2,2,3),
$g=97$}}
\scalebox{0.4}{\includegraphics*[0pt,0pt][152pt,152pt]{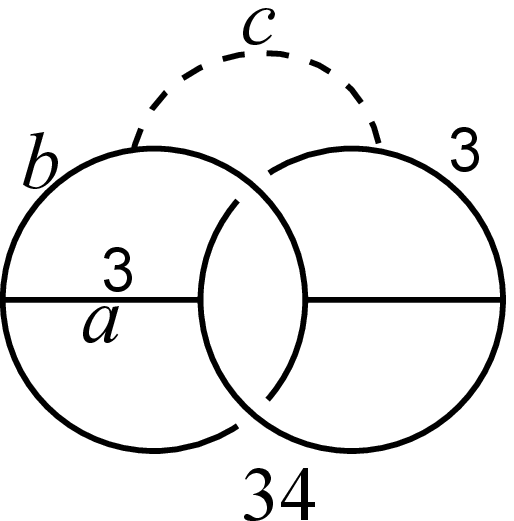}}
\raisebox{45pt} {\parbox[t]{102pt}{$|G|=120$\\$a$:(2,2,2,3),
$g=11$\\$b_{k}$:(2,2,2,3), $g=11$\\$c_{k}$: II, $g=21$}}
\end{center}

\begin{center}
\scalebox{0.4}{\includegraphics*[0pt,0pt][152pt,152pt]{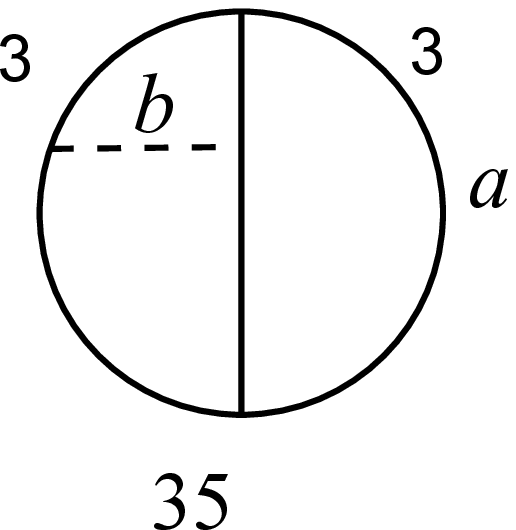}}
\raisebox{45pt} {\parbox[t]{102pt}{$|G|=12$\\$a$: I,
$g=3$\\$b_{uk}$: II, $g=3$}}
\scalebox{0.4}{\includegraphics*[0pt,0pt][152pt,152pt]{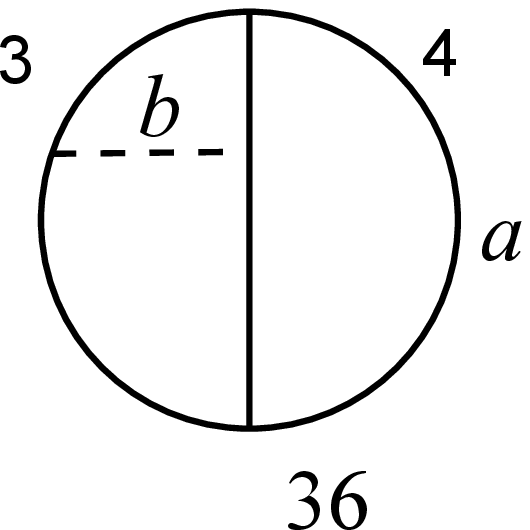}}
\raisebox{45pt} {\parbox[t]{102pt}{$|G|=24$\\$a$: I,
$g=5$\\$b_{uk}$: II, $g=5$}}
\end{center}

\begin{center}
\scalebox{0.4}{\includegraphics*[0pt,0pt][152pt,152pt]{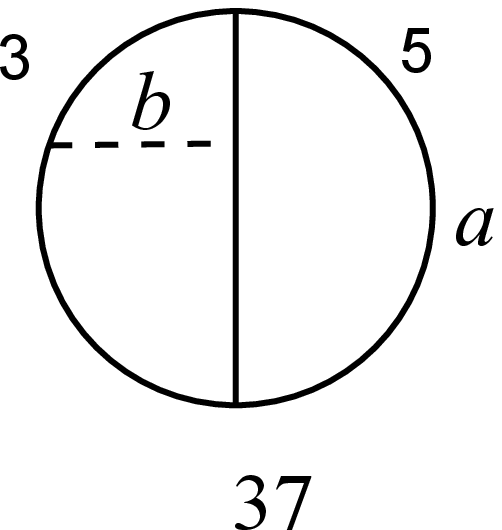}}
\raisebox{45pt} {\parbox[t]{102pt}{$|G|=60$\\$a$: I,
$g=11$\\$b_{uk}$: II, $g=11$}}
\scalebox{0.4}{\includegraphics*[0pt,0pt][152pt,152pt]{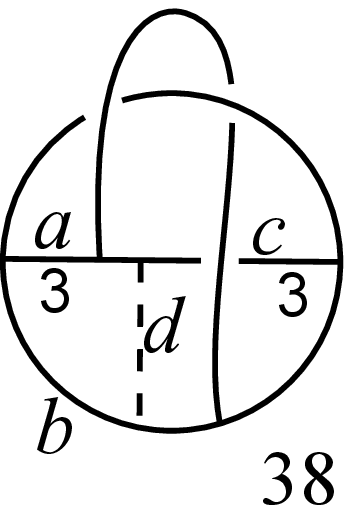}}
\raisebox{45pt} {\parbox[t]{102pt}{$|G|=2880$\\$a, b_k,
c_k$:(2,2,2,3)\\$g=241$\\$d_k$: II, $g=481$}}
\end{center}

\begin{center}
\scalebox{0.4}{\includegraphics*[0pt,0pt][152pt,152pt]{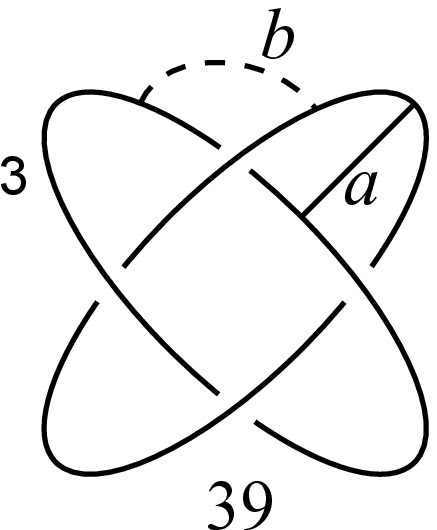}}
\raisebox{45pt} {\parbox[t]{102pt}{$|G|=576$\\$a$: I,
$g=97$\\$b_{uk}$: II, $g=97$}}
\scalebox{0.4}{\includegraphics*[0pt,0pt][152pt,152pt]{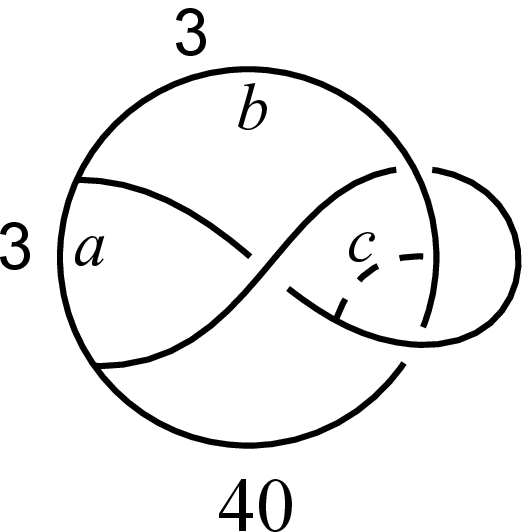}}
\raisebox{45pt} {\parbox[t]{102pt}{$|G|=1440$\\$a$: I,
$g=241$\\$b_k$: I, $g=241$\\$c_{uk}$: II, $g=241$}}
\end{center}


\section{The case of bordered surfaces: Edges in 3-orbifolds providing bordered surface embeddings with maximum symmetry}

Suppose now that $G$ is a finite group acting on some bordered surface $\Sigma$; then each singular point in the orbifold $X=\Sigma/G$ is of one of three types (see Figure 3): (a) an isolated singular point lying in the inner of $X$, corresponding to a cyclic stable subgroup; (b) a singular point lying on a reflection boundary, corresponding to a $\Z_2$ stable subgroup; (c) a corner point, corresponding to a $D_n$ dihedral stable subgroup.
\begin{center}
\centerline{ \scalebox{1.2}{\includegraphics{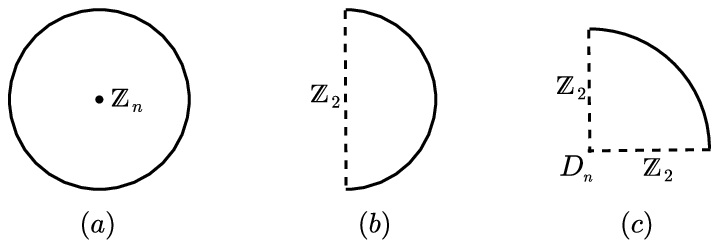}}}
Figure 3 Possible singular points of 2-orbifolds
\end{center}

We first recall an example of \cite{WWZZ}, which gives a lower bound for $m_\a$.

\begin{example}\label{general-I}
For every  $\a>1$, we will construct a
group $G$ of order $4(\a+1)$ which acts on $(S^3, \Sigma_{0,\a+1})$.
Let $\Sigma_{0,\a+1}$ be  the equator sphere $S^2$ of $S^3$ with $\a+1$
punctured holes, see Figure 4 for $\a=4$. We choose the holes all on the equator $S^1$ of $S^2$, centered at
the vertices of a regular $\a+1$-polygon. There is a dihedral group
$D_{\a+1}$ acting on $(S^3, S^2)$ which keeps $\Sigma_{0,\a+1}$ invariant. And there is also a $\Z_2$ action on $S^3$ changing the inner and outer
of $S^2$, whose fixed point set is the equator of the 2-sphere in Figure 4. So there is a $D_{\a+1}\times
\Z_2$ action on $\Sigma_{0,\a+1}$. This group has order $4(\a+1)$ and corresponds to the orbifold 15E in Table 2 of Theorem 2.1.
\begin{center}
\scalebox{0.9}{\includegraphics{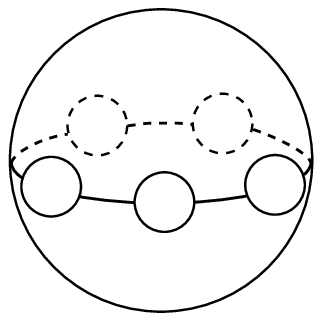}}

Figure 4 The 2-sphere with five boundary components
\end{center}

\end{example}

Suppose $G$ acts on $(S^3, \Sigma)$. Consider a $G$-equivariant regular neighborhood of $\Sigma$ which is a handlebody $V_\a$; then $G$ acts on $(S^3, V_\a)$. By the previous example, if this action realizes $m_\a$, then we have $|G| \geq 4(\a+1)>4(\a-1)$.
So as discussed in Section 2, $V_\a/G$ can be of two types as in Figure 2.

If $V_\a/G$ is of type (a) then since $V_\a/G$ is a regular neighborhood of $X$, the two degree-$3$ singular points (i.e., valence three vertices of the singular trivalent graphs) together with the middle singular edge joining them must lie in $X$. So the singular edge corresponds to a reflection boundary of $X$, and the two degree-$3$ singular points correspond to two corner points of $X$. So in this case, $X$ is a disk with three reflection boundaries and two corner points. Consider $\partial X = \partial \Sigma /G$; this boundary arc is isotopy to an arc on the boundary of $V_\a/G$, and it is an arc joining two singular points on $\partial V_\a/G$. $\partial V_\a/G$ is a sphere with four singular points, and we are now considering an arc joining two of them and the arc is disjoint from the other two. Since a sphere minus two points has fundamental group $\Z$, the isotopy class of  $\partial X$ is in the sequence indicated in Figure 5(a).

If $V_\a/G$ is of type (b) then the index-$2$ singular arc may be either a reflection boundary of $X$ or intersect $X$ in an isolated point, and the index-$3$ singular point must be isolated in $X$. So there are two cases as shown in Figure 5(b).

\begin{center}
\scalebox{0.6}{\includegraphics{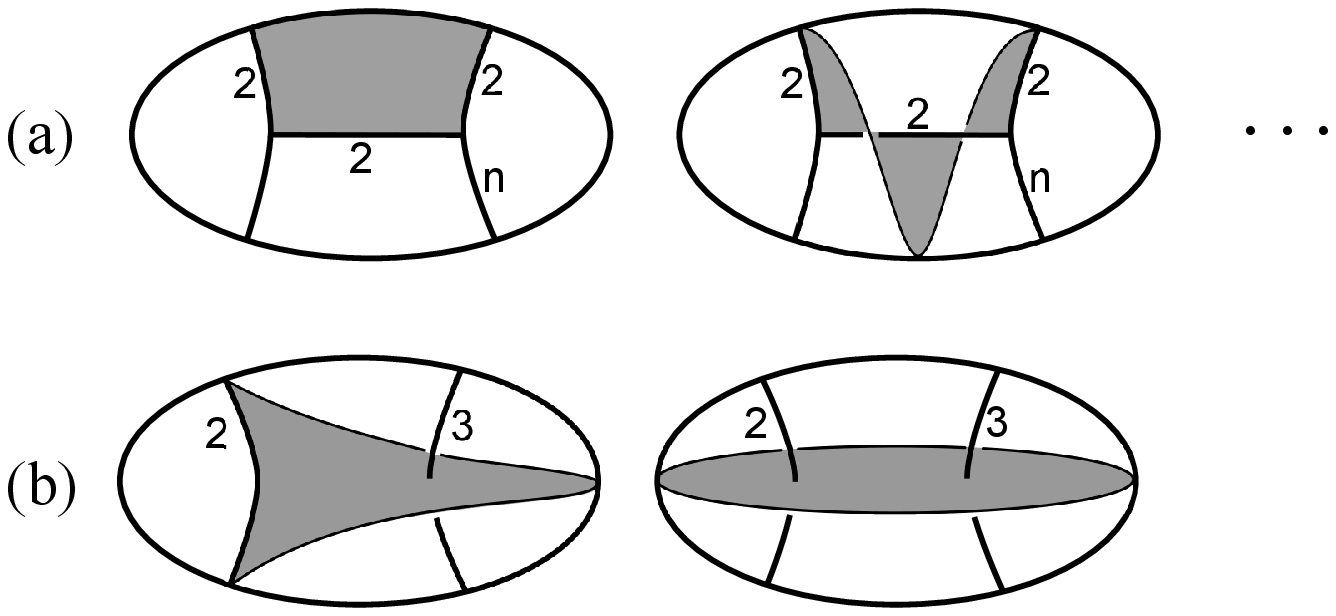}}

Figure 5 Possible embeddings of 2-orbifolds into 3-orbifolds
\end{center}

From the discussion above, not all the allowable orbifolds given in Theorem \ref{classify} corresponds to some bordered surface orbifold $X$. For type (I), we have seen that if the singular edge corresponds to some $X$, then it must have index $2$, and the stable subgroups corresponding to its two ends must be dihedral groups.

Call two singular edges/dashed arcs \textit{equivalent} if there is an
orbifold automorphism sending one to the other. However, if two (nonequivalent) singular edges/dashed arcs have regular neighborhoods with isotopic boundaries (so that an unknotted $2$-orbifold splits the spherical $3$-orbifold into two handlebody orbifolds), then we say that the edges/arcs are \textit{dual} to each other.

In the study of maximum symmetry of closed surfaces, it is
reasonable to call two dual edges equivalent since they produce
the same allowable $2$-orbifold. But it is not good to call them
equivalent in the study of maximum symmetry of bordered surfaces
since they may correspond to different $2$-orbifolds or to
different bordered surfaces.

Then by a routine checking of Theorem \ref{classify}, we have the following Theorem \ref{position}
for our further study of the realizations of the maximum symmetry of
$(S^3, \Sigma)$ in the following sense: (1) we only pick information
from Theorem \ref{classify} related to $m_\a$ (so among 40 orbifolds
listed in Theorem \ref{classify}, only 16 orbifolds appear in
Theorem \ref{position});  and list the orbifolds according to the
sizes of $m_\a$ (so one figure in Theorem \ref{classify} can become
several figures in Theorem \ref{position}, for example, 15E or 27). (2) In Table 4 of Theorem 3.2, $a'$, $b'$, ... denote edges which are dual to the edges $a$, $b$, ... in the tables of Theorem 2.1; for example, the edge $a$ of orbifold 26 in Table 3 of Theorem 2.1 is not allowable since it has singular order three but its dual edge $a'$ has singular order two and is allowable (see Table 4 of Theorem 3.2).

The topological type of the bordered surfaces are also listed in the table. The method to determine their topological genus and the number of boundary components will be explained in the next section.

\begin{theorem}\label{position}
Up to equivalent embeddings, all the edges/arcs corresponding to
bordered surfaces with maximum symmetry are listed below (for dashed arcs only one position is showed).
\end{theorem}
\begin{table}[h]
\caption{Allowable edges/arcs corresponding to
bordered surfaces with maximum symmetry}\label{tab:allow}
\end{table}
\centerline{$|G|=12(\a-1)$}
\begin{center}
\scalebox{0.4}{\includegraphics*[0pt,0pt][152pt,152pt]{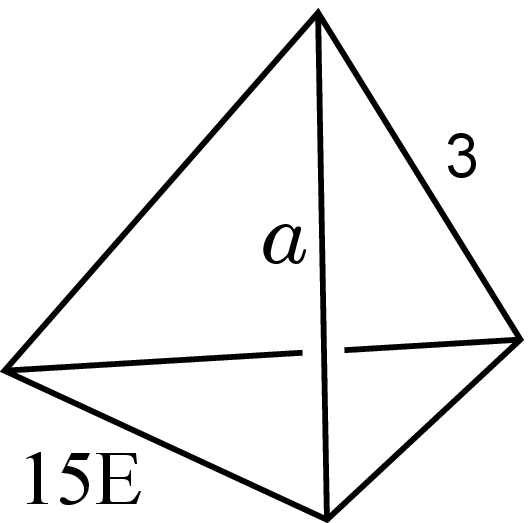}}
\hspace{10pt}\raisebox{35pt} {\parbox[t]{102pt}{$|G|=12$\\$\a=2\\ \Sigma_{0,3},\Sigma_{1,1}$}}
\scalebox{0.4}{\includegraphics*{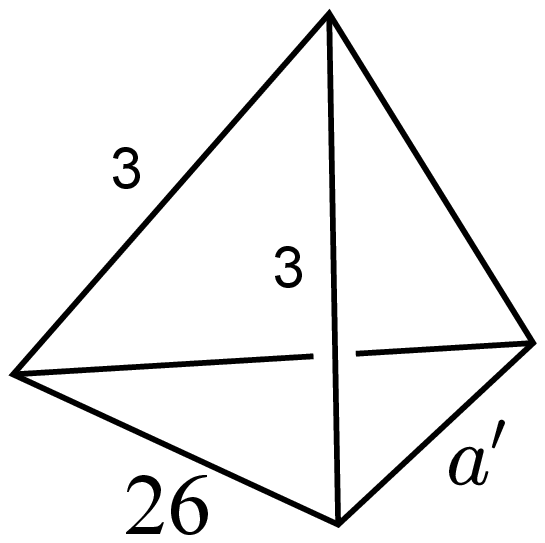}}
\hspace{10pt}\raisebox{35pt} {\parbox[t]{102pt}{$|G|=24$\\$\a=3$\\$\Sigma_{0,4},\Sigma^-_{1,3}$}}
\end{center}

\begin{center}
\scalebox{0.4}{\includegraphics*[0pt,0pt][152pt,152pt]{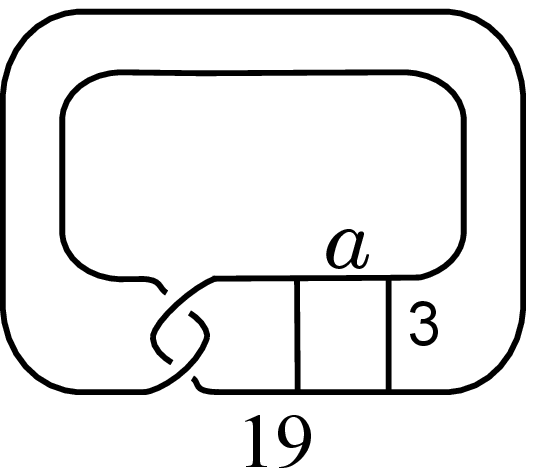}}
\hspace{10pt}\raisebox{35pt} {\parbox[t]{102pt}{$|G|=36$\\$\a=4\\\Sigma_{1,3}$}}
\scalebox{0.4}{\includegraphics*{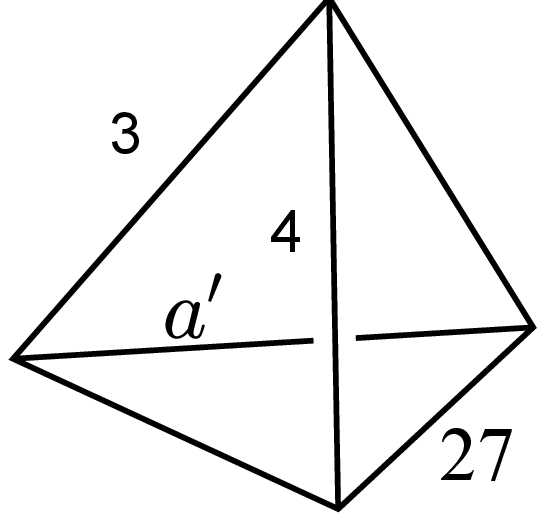}}
\hspace{10pt}\raisebox{35pt} {\parbox[t]{102pt}{$|G|=48$\\$\a=5\\ \Sigma_{0,6}, \Sigma_{1,4}$}}
\end{center}

\begin{center}
\scalebox{0.4}{\includegraphics*[0pt,0pt][152pt,152pt]{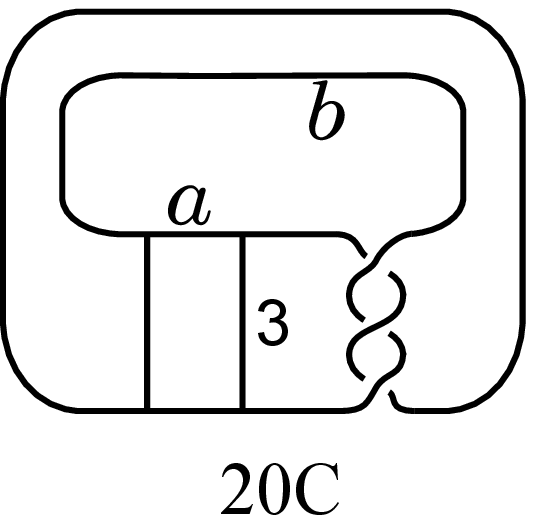}}
\hspace{10pt}\raisebox{35pt} {\parbox[t]{102pt}{$|G|=96$\\$\a=9\\ \Sigma_{2,6}, \Sigma_{3,4}$}}
\scalebox{0.4}{\includegraphics*{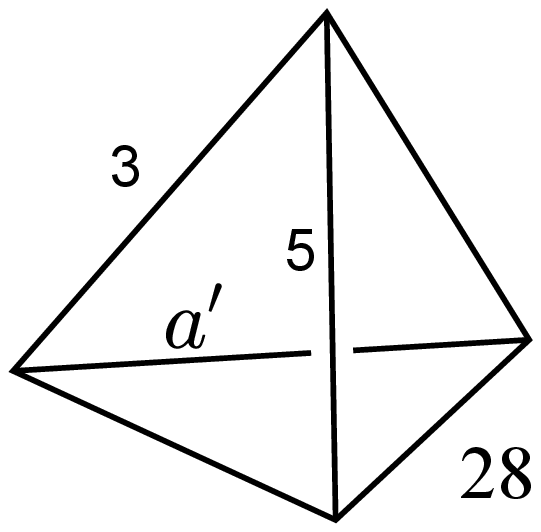}}
\hspace{10pt}\raisebox{35pt} {\parbox[t]{102pt}{$|G|=120$\\$\a=11\\\Sigma_{0,12},\Sigma^-_{6,6}$}}
\end{center}

\begin{center}
\scalebox{0.4}{\includegraphics*{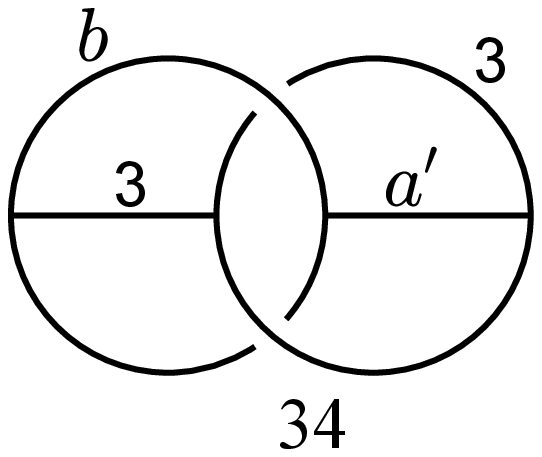}}
\hspace{10pt}\raisebox{35pt} {\parbox[t]{102pt}{$|G|=120$\\$\a=11\\\Sigma_{0,12},\Sigma^-_{6,6}$}}
\scalebox{0.4}{\includegraphics*[0pt,0pt][152pt,152pt]{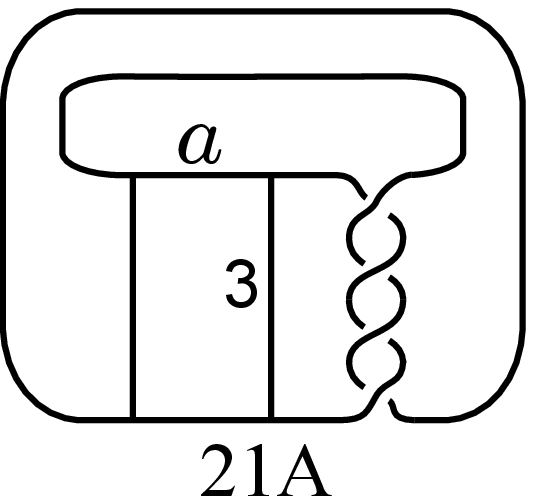}}
\hspace{10pt}\raisebox{35pt} {\parbox[t]{102pt}{$|G|=288$\\$\a=25\\\Sigma_{7,12},\Sigma_{10,6}$}}
\end{center}

\begin{center}
\scalebox{0.4}{\includegraphics*{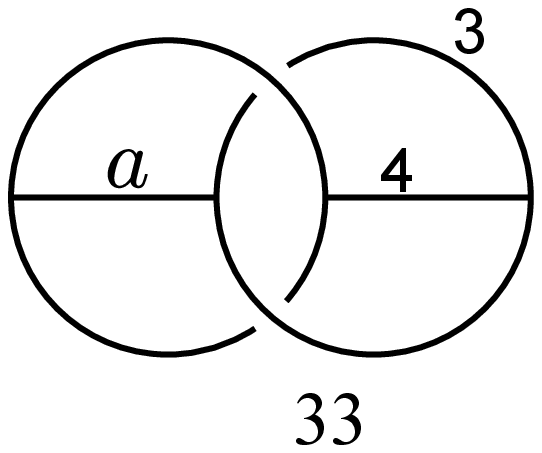}}
\hspace{10pt}\raisebox{35pt} {\parbox[t]{102pt}{$|G|=1152$\\$\a=97\\\Sigma_{37,24}$}}
\scalebox{0.4}{\includegraphics*[0pt,0pt][152pt,152pt]{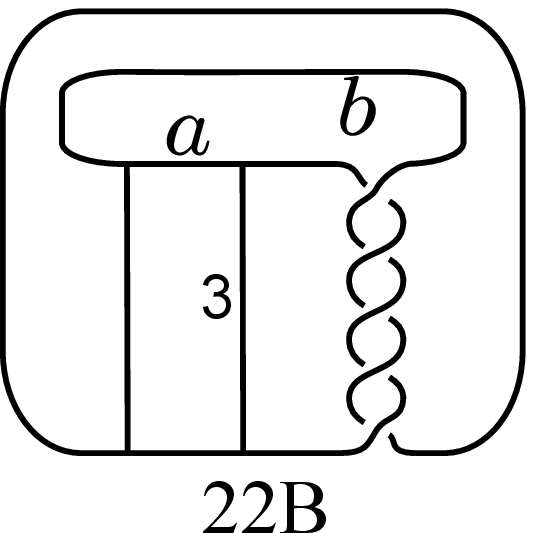}}
\hspace{10pt}\raisebox{35pt}
{\parbox[t]{102pt}{$|G|=1440$\\$\a=121\\\Sigma_{43,36},\Sigma_{55,12}$}}
\end{center}

\begin{center}
\scalebox{0.4}{\includegraphics*{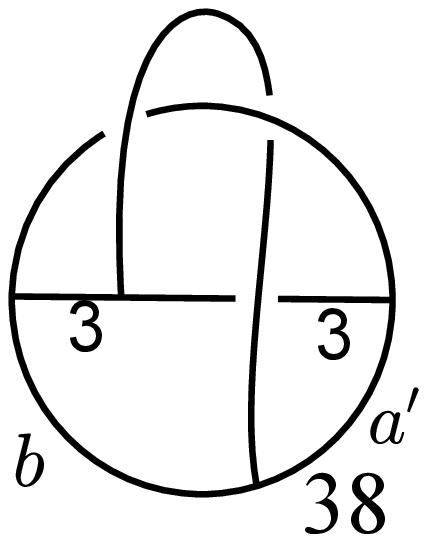}}
\hspace{10pt}\raisebox{35pt}
{\parbox[t]{102pt}{$|G|=2880$\\$\a=241\\ \Sigma_{73,96}, \Sigma_{97,48},\Sigma^-_{206,36}$}}\hspace*{180pt}
\end{center}

\centerline{$|G|=8(\a-1)$}
\begin{center}
\scalebox{0.4}{\includegraphics*{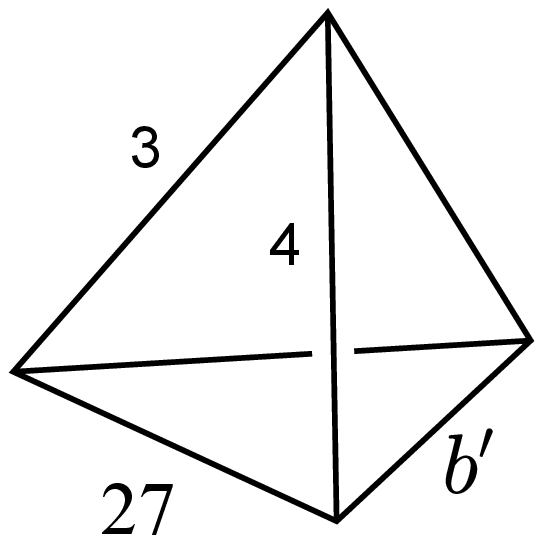}}
\hspace{10pt}\raisebox{35pt} {\parbox[t]{102pt}{$|G|=48$\\$\a=7\\ \Sigma_{0,8},\Sigma^-_{4,4}$}}
\scalebox{0.4}{\includegraphics*[0pt,0pt][152pt,152pt]{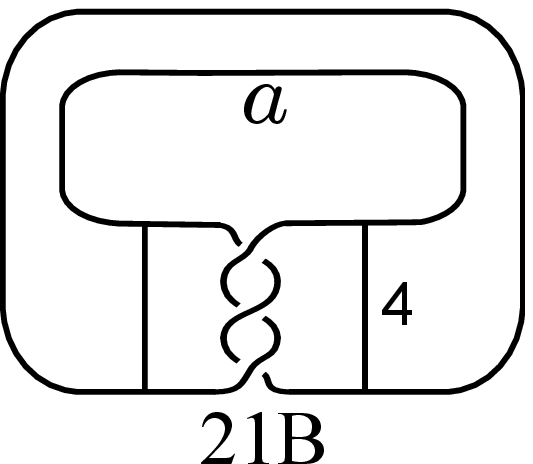}}
\hspace{10pt}\raisebox{35pt} {\parbox[t]{102pt}{$|G|=384$\\$\a=49\\\Sigma_{17,16},\Sigma_{21,8}$}}
\end{center}

\centerline{$|G|=20(\a-1)/3$}
\begin{center}
\scalebox{0.4}{\includegraphics*[0pt,0pt][152pt,152pt]{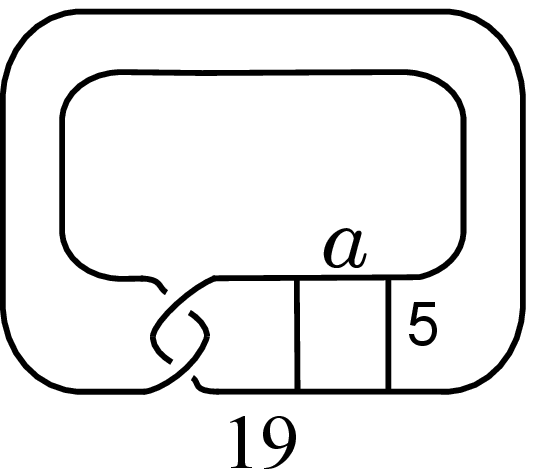}}
\hspace{10pt}\raisebox{35pt} {\parbox[t]{102pt}{$|G|=100$\\$\a=16\\\Sigma_{6,5}$}}
\scalebox{0.4}{\includegraphics*{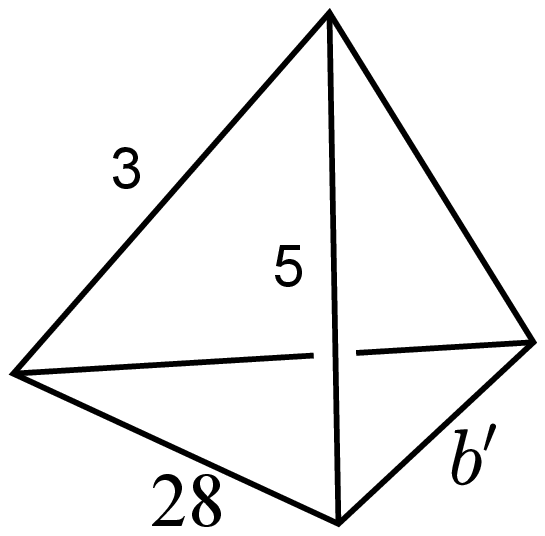}}
\hspace{10pt}\raisebox{35pt} {\parbox[t]{102pt}{$|G|=120$\\$\a=19\\\Sigma_{0,20},\Sigma^-_{14,6}$}}
\end{center}

\begin{center}
\scalebox{0.4}{\includegraphics*[0pt,0pt][152pt,152pt]{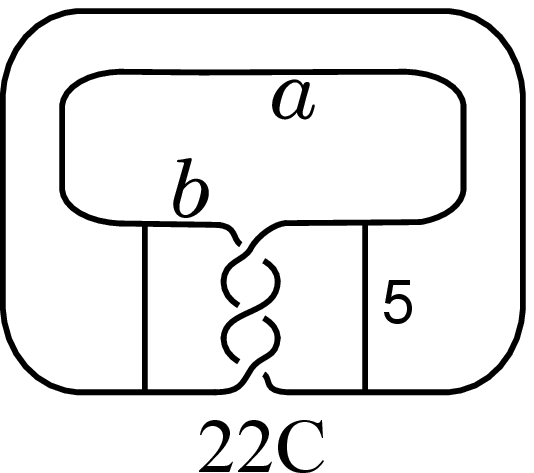}}
\hspace{10pt}\raisebox{35pt}
{\parbox[t]{102pt}{$|G|=2400$\\$\a=361\\\Sigma_{131,100},\Sigma_{151,60},\Sigma_{171,20}$}}\hspace*{180pt}
\end{center}

\centerline{$|G|=6(\a-1)$}
\begin{center}
\scalebox{0.4}{\includegraphics*{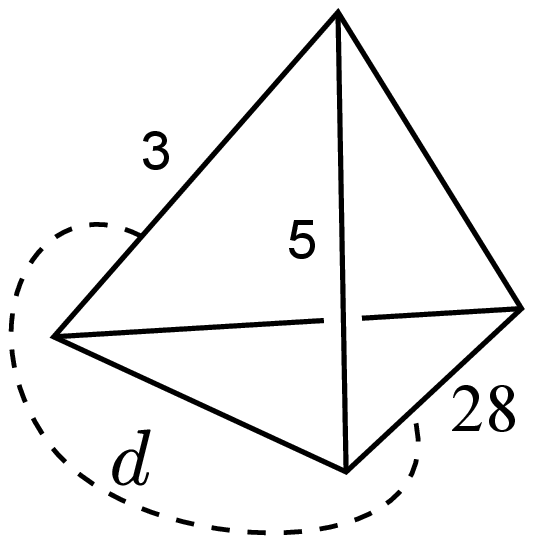}}
\hspace{10pt}\raisebox{35pt} {\parbox[t]{102pt}{$|G|=120$\\$\a=21\\\Sigma_{5,12}$}}
\scalebox{0.4}{\includegraphics*{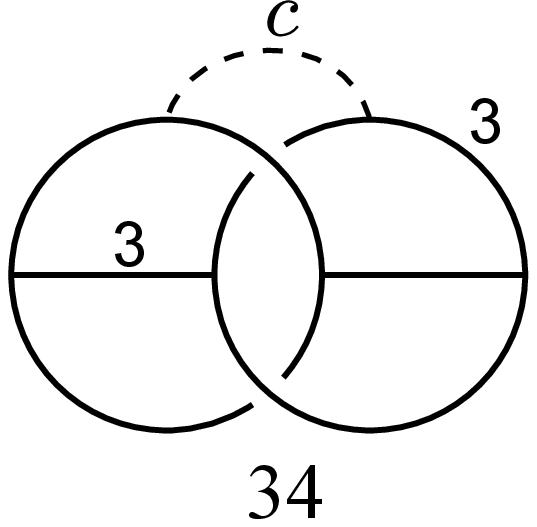}}
\hspace{10pt}\raisebox{35pt} {\parbox[t]{102pt}{$|G|=120$\\$\a=21\\\Sigma_{5,12}$}}
\end{center}

\begin{center}
\scalebox{0.4}{\includegraphics*{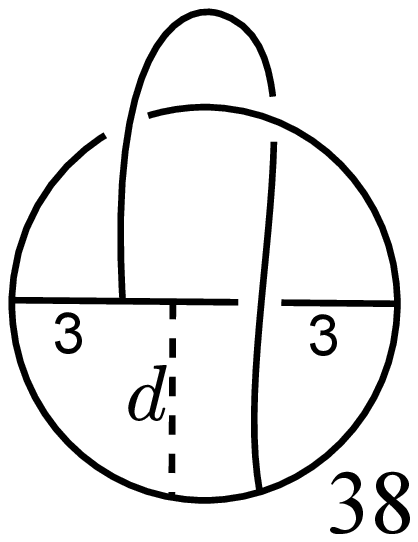}}
\hspace{10pt}\raisebox{35pt}
{\parbox[t]{102pt}{$|G|=2880$\\$\a=481\\\Sigma_{205,72},\Sigma_{193,96}$}}\hspace*{180pt}
\end{center}

\centerline{$|G|=24(\a-1)/5$ and $30(\a-1)/7$}
\begin{center}
\scalebox{0.4}{\includegraphics*{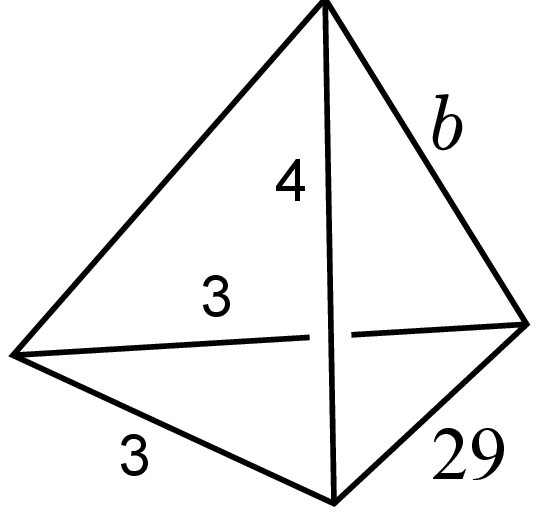}}
\hspace{10pt}\raisebox{35pt} {\parbox[t]{102pt}{$|G|=192$\\$\a=41\\\Sigma^-_{30,12}$}}
\scalebox{0.4}{\includegraphics*{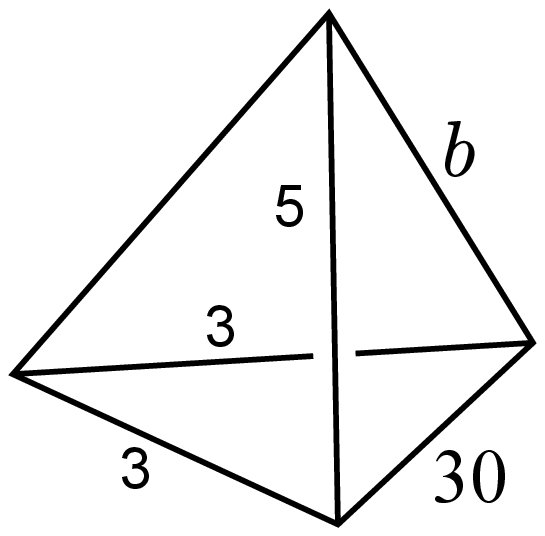}}
\hspace{10pt}\raisebox{35pt}
{\parbox[t]{102pt}{$|G|=7200$\\$\a=1681\\\Sigma^-_{1562,120}$}}
\end{center}

\centerline{$|G|=4(\sqrt{\a}+1)^2$, $g=k^2, k\neq3, 5, 7, 11, 19,
41$}
\begin{center}
\scalebox{0.4}{\includegraphics*[0pt,0pt][152pt,152pt]{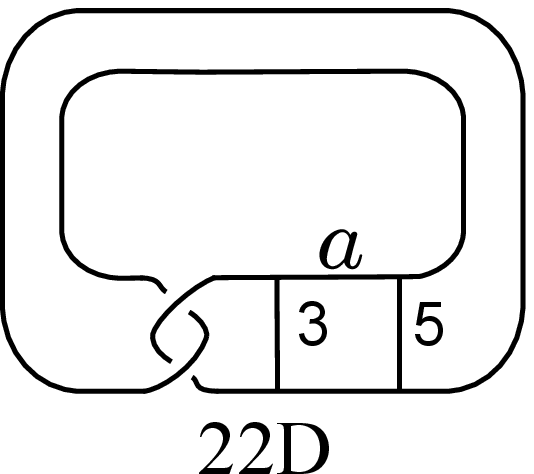}}
\hspace{10pt}\raisebox{35pt}
{\parbox[t]{102pt}{$|G|=3600$\\$\a=841\\\Sigma_{391,60},\Sigma_{406,30}$}}
\scalebox{0.4}{\includegraphics*[0pt,0pt][152pt,152pt]{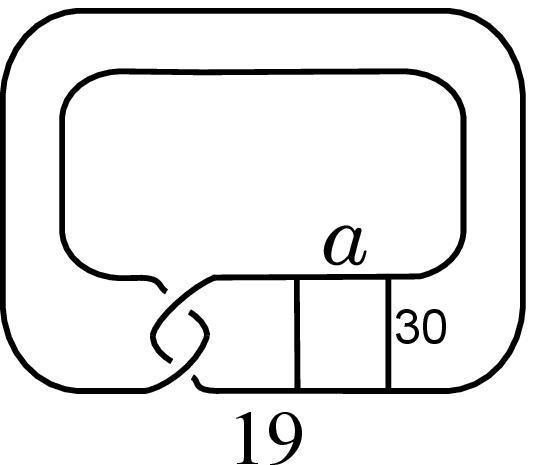}}
\hspace{10pt}\raisebox{35pt}
{\parbox[t]{102pt}{$|G|=3600$\\$\a=841\\\Sigma_{391,60}$}}
\end{center}

\begin{center}
\scalebox{0.4}{\includegraphics*[0pt,0pt][152pt,152pt]{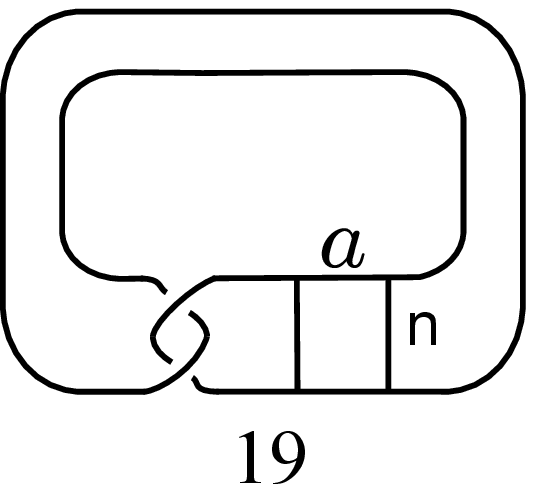}}
\hspace{10pt}\raisebox{35pt}
{\parbox[t]{102pt}{Others, $|G|=4n$\\$\a=(n-1)^2$\\$\Sigma_{\frac{(n-1)(n-2)}{2},n}$}}\hspace*{180pt}
\end{center}

\centerline{$|G|=4(\a+1)$, $\a$ belongs to remaining numbers}
\begin{center}
\scalebox{0.4}{\includegraphics*{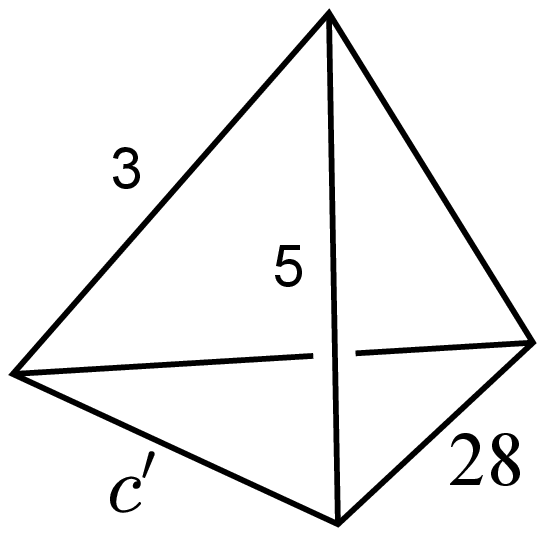}}
\hspace{10pt}\raisebox{35pt} {\parbox[t]{102pt}{$|G|=120$\\$\a=29\\\Sigma_{0,30},\Sigma_{9,12}$}}
\scalebox{0.4}{\includegraphics*[0pt,0pt][152pt,152pt]{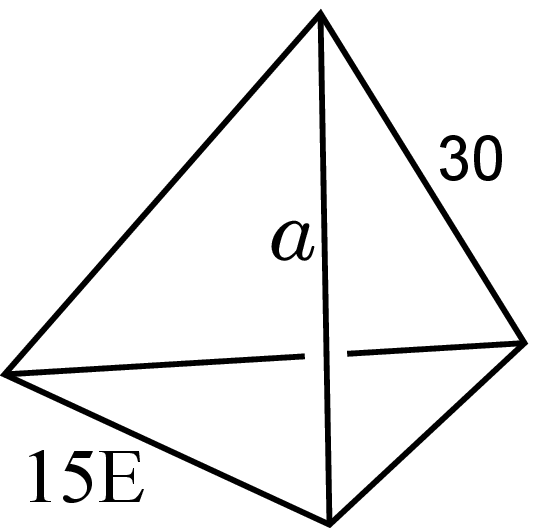}}
\hspace{10pt}\raisebox{35pt} {\parbox[t]{102pt}{$|G|=120$\\$\a=29\\\Sigma_{14,2}$}}
\end{center}

\begin{center}
\scalebox{0.4}{\includegraphics*[0pt,0pt][152pt,152pt]{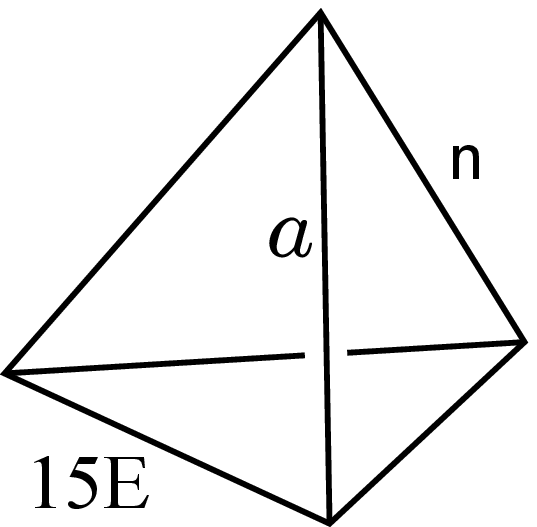}}
\hspace{10pt}\raisebox{35pt}
{\parbox[t]{102pt}{Others, $|G|=4n$\\$\a=n-1$\\
$\Sigma_{0,n}$, $\Sigma_{\frac{n-1}{2},1}(n \,\,\text{odd})$, $\Sigma_{\frac{n-2}{2},2}(n \,\,\text{even})$}}\hspace*{180pt}
\end{center}

\section {The bordered surfaces realizing the maximum symmetry}
We use the following lemma from \cite{WWZZ2} to compute the number of boundary components of the preimage of $X$.

\begin{lemma}[Lemma 2.7 \cite{WWZZ2}]\label{boundNums}
Suppose that $M$ is simply connected with an action of $G$, and that $X$ is a suborbifold of $M/G$ such that $|X|$ is connected. Let $i: X\hookrightarrow M/G$ be the inclusion map; then the preimage of $X$ in $M$ has $[\pi_1(M/G):i_*(\pi_1(X))]$ connected components.
\end{lemma}

Then we use the following lemmas to determine whether the preimage of $X$ is oriented.

\begin{lemma}[Lemma 2.5 \cite{WWZZ2}]\label{oriented1}
Suppose that $G$ acts on a compact surface $\Sigma$ such that each singular point in the orbifold $X=\Sigma/G$ is isolated and the underlying space $|X|$ is orientable; then $\Sigma$ is orientable.
\end{lemma}

\begin{lemma}\label{oriented2}
Suppose that $G$ acts on a compact surface $\Sigma$ such that $X=\Sigma/G$ has at least one reflection boundary and the underlying space $|X|$ is orientable. Then $\Sigma$ is oriented if and only if there is a group homomorphism $h: G\rightarrow \Z_2$ sending each element corresponding to a reflection boundary of $X$ to the non-trivial element in $\Z_2$.
\end{lemma}
\begin{proof}
If $\Sigma$ is oriented, let $G^o$ be the subgroup of $G$ which contains all elements preserving the orientation of $\Sigma$; then $G/G^o \cong \Z_2$ is the desired homomorphism. Conversely, if $h$ is such a homomorphism, then consider a double cover $\tilde{X}$ of $X$ corresponding to the subgroup of the kernel of $h$. $\tilde{X}$ has only isolated singular points, so by the previous lemma $\Sigma$ is oriented.
\end{proof}

So using these lemmas, we check all the allowable $X$ coming from allowable singular edges/dashed arcs in Table 4 of Theorem 2.2 and determine the number of boundary components, orientability and the topological genus of their preimages in $S^3$.

\begin{proof}[Proof of Theorem \ref{mg}]
Note that in Table \ref{tab:allow} only two orbifolds $15$E and $19$ have a free parameter $n$; for these two orbifolds, we check the preimages of possible embeddings of $X$ as follows.

In the orbifold $15$E, we lift the embedding of $X$ to a double cover $S^3/\widetilde{G}$, where $\widetilde{G}$ is an abelian group of order $2n$:
$$\widetilde{G}=\langle x,y \mid x^2, y^n, xyx^{-1}y^{-1}\rangle.$$
Here $x$ and $y$ are the generators corresponding to the index $2$ and index $n$ singular circles in $S^3/\widetilde{G}$.

One possible embedding of $X$ is as in Figure 6; by lifting to a double cover, the fundamental group of the boundary of $\widetilde{X}$ in $S^3/\widetilde{G}$ corresponds to the subgroup generated by $x$ and $yxy^{-1}=x$ (the image of the fundamental group of the orbifold corresponding to the boundary of $\widetilde{X}$). So this subgroup has index $n$ in $\widetilde{G}$, which means $\Sigma$ has $n$ boundary components by Lemma \ref{boundNums}. Furthermore, in this case,
$$\pi_1(\widetilde{X})=\langle x, y\rangle$$
where $x$ corresponds to the reflection boundary and $y$ corresponds to the isolated singular point in $\widetilde{X}$. If we take
$$\Z_2\cong\langle t \mid t^2\rangle,$$
then $h: \widetilde{G} \rightarrow \Z_2: (x,y)\mapsto(t, e)$ is a homomorphism. So by Lemma \ref{oriented2}, $\Sigma$ is oriented. Then by the relation $\a=2g-1+b$, we have $\Sigma = \Sigma_{0,n}$.
\begin{center}
\scalebox{0.6}{\includegraphics{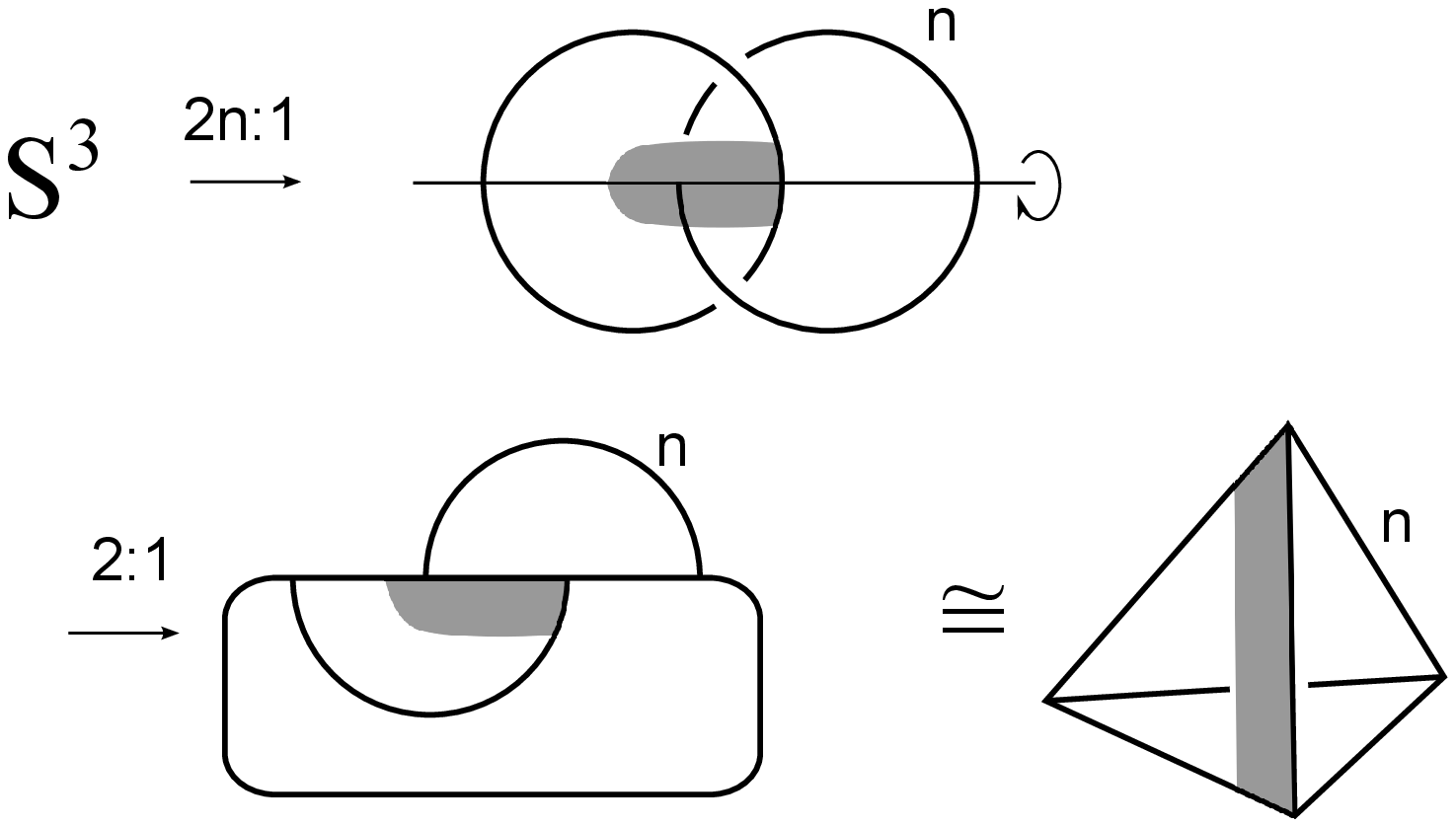}}

Figure 6 Embeddings of 2-orbifolds into 3-orbifolds
\end{center}
Another possible embedding of $X$ is as in Figure 7; by lifting to a double cover, the image of the fundamental group of the boundary of $\widetilde{X}$ in $S^3/\widetilde{G}$ corresponds to the subgroup generated by $xy$. This subgroup has index $1$ if $n$ is odd, and index $2$ is $n$ is even. By Lemma \ref{oriented1}, $\Sigma$ is oriented, so $\Sigma= \Sigma_{\frac{n-1}{2},1}$ for $n$ odd and $\Sigma= \Sigma_{\frac{n-2}{2},2}$ for $n$ even.
\begin{center}
\scalebox{0.6}{\includegraphics{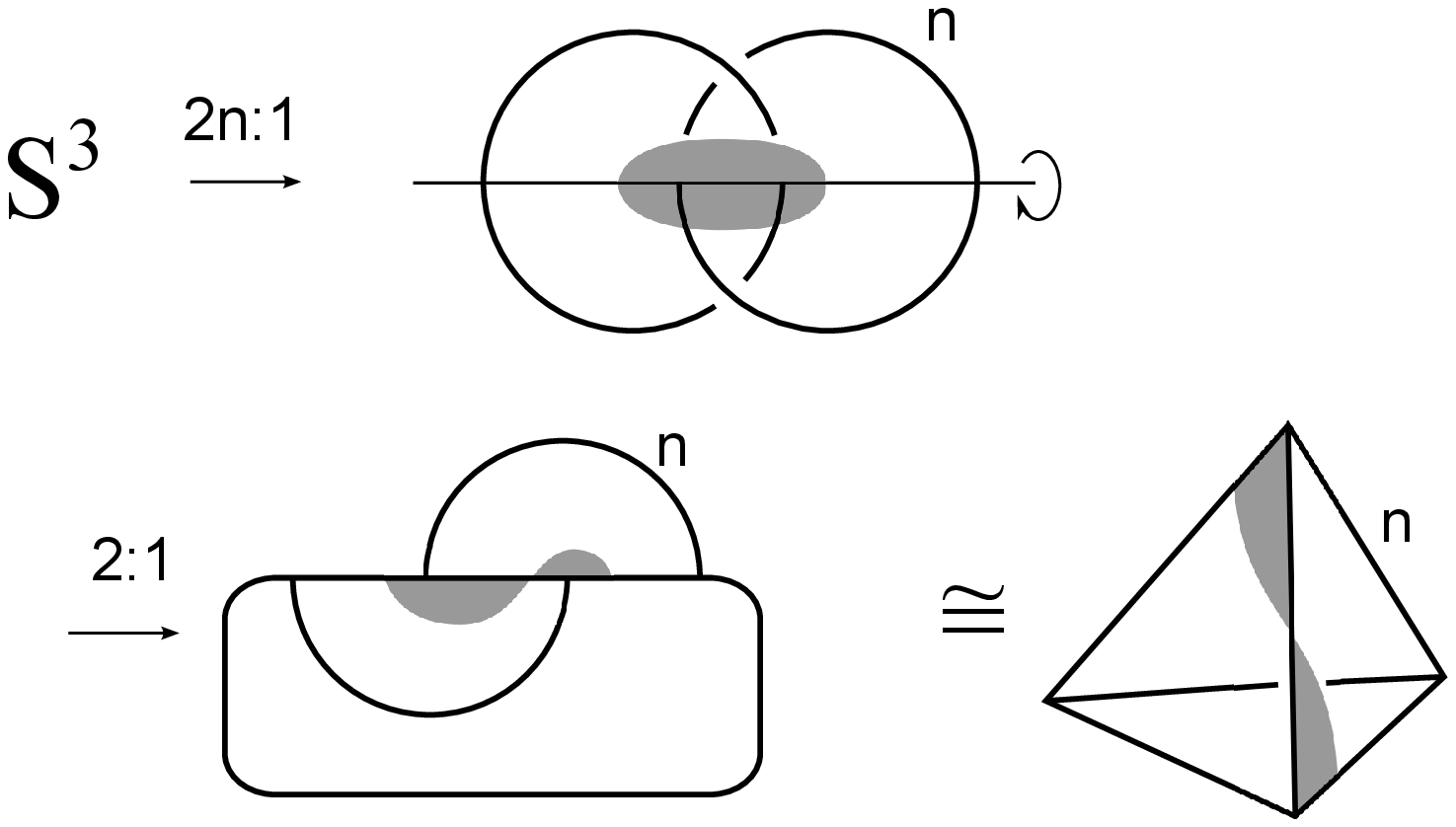}}

Figure 7 Embeddings of 2-orbifolds into 3-orbifolds
\end{center}
For other possible embeddings of $X$ in the orbifold $15$E (we mean in the sequence given in Figure 5(a)), it is easy to see $\Sigma$ is diffeomorphic to one case discussed above.

In the orbifold $19$, one possible embedding of $X$ is as in Figure 8. By lifting to a $4$-sheet cover, the boundary of $\widetilde{X}$ in $S^3/\widetilde{G}$ corresponds to the subgroup generated by $xy$, where
$$\widetilde{G}=\langle x,y \mid x^n, y^n, xyx^{-1}y^{-1}\rangle.$$
 This subgroup has index $n$. By Lemma \ref{oriented1}, $\Sigma$ is oriented. So $\Sigma= \Sigma_{\frac{(n-1)(n-2)}{2},n}$. For other possible embeddings of $X$ in the orbifold $19$, it is easy to see that we get the same $\Sigma$.
\begin{center}
\scalebox{0.6}{\includegraphics{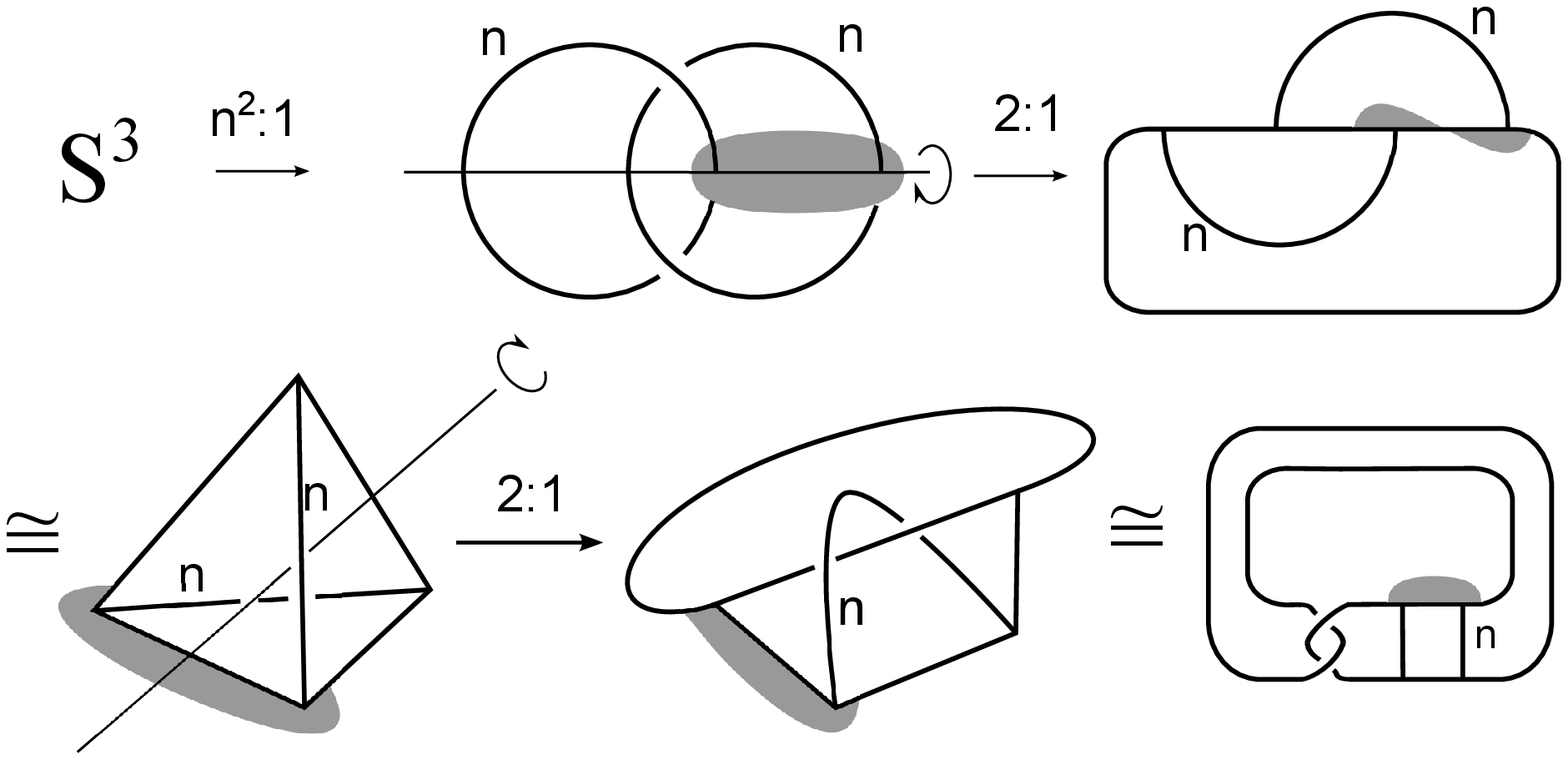}}

Figure 8 Embeddings of 2-orbifolds into 3-orbifolds
\end{center}

We also remark that in these orbifolds, one can also directly lift $X$ to $S^3$ to see which surface $\Sigma$ is. But we think it is a good example here to show how the Lemmas given above work. In all other orbifolds given in Table \ref{tab:allow}, there are no parameters and only finitely many groups $G$, so we can check the situation by some computer software case by case. We present two such examples, one for the case of a singular edge and one for a dashed arc; in all other cases, the computations are similar.
\end{proof}

\begin{example}[Singular edge case]
Let us show how we use \cite{GAP} to determine the bordered surface type for the singular edge $a'$ in orbifold $28$. First we use the Wirtinger presentation with generators as in Figure 9 to get

\begin{center}
\scalebox{0.8}{\includegraphics{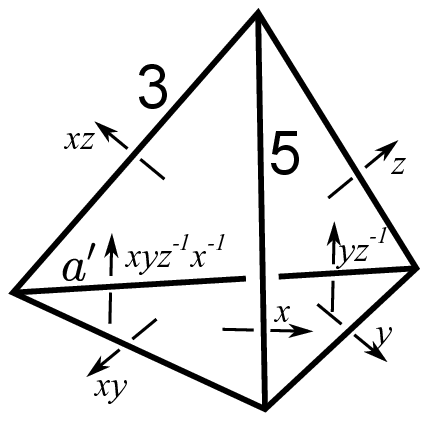}}

Figure 9 Generators of the fundamental group of a tetrahedral 3-orbifold
\end{center}

$$G=\pi_1(\mathcal{O}_{28})=\langle x,y,z\mid
x^5,y^2,z^2,(xz)^3,(xy)^2,(yz^{-1})^2\rangle.$$
Then consider the possible embeddings of $X$ with $a'$ as one of its reflection edge. The other reflection edge on the left corner must correspond to the edge labeled $xy$, and the other reflection edge on the right corner may correspond to the edge labeled $y$ or $z$. Note also that the embedding of $X$ may have some twist around $a'$, so the subgroup corresponding to the boundary of $X$ is one of the following:
$$G_1 = \langle xy, xyx^{-1}\rangle \subset G,$$
$$G_2 = \langle xy, xzx^{-1}\rangle \subset G,$$
$$G_3 = \langle xy, (xyz^{-1})y(xyz^{-1})^{-1}\rangle \subset G,$$
$$G_4 = \langle xy, (xyz^{-1})z(xyz^{-1})^{-1}\rangle \subset G.$$
Then we use \cite{GAP} to compute the index $G:G_i$ to get the number of boundary components: $|G:G_1|=|G:G_2|=12$ and $|G:G_3|=|G:G_4|=6$. Also we use Lemma \ref{oriented2} to determine if $\Sigma$ is oriented: the surfaces corresponding to $G_1$ and $G_2$ are oriented, the surfaces corresponding to $G_3$ and $G_4$ are non-oriented. So we get $\Sigma_{0,12}$ and $\Sigma^-_{6,6}$ for this case. Codes for \cite{GAP} are listed here.

\begin{lstlisting}
f := FreeGroup( "x","y","z");
G := f / [f.1^5, f.2^2, f.3^2, (f.1*f.3)^3, (f.1*f.2)^2,
    (f.2*f.3^-1)^2];
Print(Size(G),"\n"); #120
f2 := FreeGroup( "t"); Z2 := f2 / [f2.1*f2.1]; x:=G.1;
y:=G.2; z:=G.3; midarc:=x*y*z^-1*x^-1; rightArc1:=x*y*x^-1;
rightArc2:=x*z*x^-1; leftArc:=x*y; algebGenus := 11;
iso1:=IsomorphismPermGroup(G); iso2:=IsomorphismPermGroup(Z2);
g1:=Image(iso1); z2:=Image(iso2); img1:=Image(iso1,midarc);
img2:=Image(iso1,rightArc1); img3:=Image(iso1,leftArc);
img22:=Image(iso1, rightArc2); img4:=Image(iso2,Z2.1); l:=[];
if GroupHomomorphismByImages(g1,z2,[img1,img2,img3],[img4,img4,img4])<>fail then
	Print("oriented!\n"); oriented:=1;
else
	Print("non-oriented!\n"); oriented:=0;
fi;
borderSubG := GroupWithGenerators([rightArc1,leftArc]);
b:=Index(G, borderSubG);
if oriented=1 then
	g:=(algebGenus+1-b)/2;
else
	g:=algebGenus+1-b;
fi;
Add(l, [g,b,oriented]);
borderSubG := GroupWithGenerators([rightArc1,midarc*leftArc*midarc^-1]);
b:=Index(G, borderSubG);
if oriented=1 then
	g:=(algebGenus+1-b)/2;
else
	g:=algebGenus+1-b;
fi;
Add(l, [g,b,oriented]);
if GroupHomomorphismByImages(g1,z2,[img1,img22, img3],[img4,img4,img4])<>fail then
	Print("oriented!\n"); oriented:=1;
else
	Print("non-oriented!\n"); oriented:=0;
fi;
borderSubG := GroupWithGenerators([rightArc2,leftArc]);
b:=Index(G, borderSubG);
if oriented=1 then
	g:=(algebGenus+1-b)/2;
else
	g:=algebGenus+1-b;
fi;
Add(l, [g,b,oriented]);
borderSubG := GroupWithGenerators([rightArc2,midarc*leftArc*midarc^-1]);
b:=Index(G, borderSubG);
if oriented=1 then
	g:=(algebGenus+1-b)/2;
else
	g:=algebGenus+1-b;
fi;
Add(l, [g,b,oriented]);
Print(l, "\n");
\end{lstlisting}
\end{example}
\noindent \textbf{Remark.}
In this case the group $G$ is $A_5 \times \Z_2$ of order $120$, and all computations can be easily done by hand, by working with permutations and cycles in $A_5$ (and similarly for various other small orders). However, for larger orders it is convenient to use some computer algebra.

\begin{example}[Dashed arc case]
We use again the orbifold $28$, considering now the embeddings of $X$ corresponding to the dashed arc as shown in Figure 10.

\begin{center}
\scalebox{0.5}{\includegraphics{g21_1.eps}}

Figure 10 Dashed arc in a tetrahedral 3-orbifold
\end{center}
We use the same Wirtinger presentation as in the previous example.
The dashed arc may be any arc in the orbifold with the same endpoints as that in Figure 10. So the subgroup corresponding to some embedding of $X$ should be generated by $y$ and $cxzc^{-1}$, here $c$ may be any element in $G$. \cite{GAP} can systematically create all elements $c$ of the group $G$ by some standard procedure. For each element $c\in G$ we first need to check if $y$ and $cxzc^{-1}$ generates the whole group of $G$, which means the preimage of $X$ is connected. Then if $X$ is as Figure 5(b) left, then the subgroup corresponding to the boundary of $X$ is generated by $\langle y, (cxzc^{-1})y(cxzc^{-1})^{-1} \rangle$ or $\langle y, (c(xz)^{-1}c^{-1})y(c(xz)^{-1}c^{-1})^{-1} \rangle$. If $X$ is as Figure 5(b) right, then the subgroup corresponding to the boundary of $X$ is generated by $\langle y(cxzc^{-1}) \rangle$ or $\langle y(cxzc^{-1})^{-1} \rangle$. Then we use Lemma \ref{boundNums} to determine the number of boundary components by the index corresponding to their subgroups: in fact, all these subgroups have index $12$.. We use Lemma \ref{oriented1} or \ref{oriented2} to determine if it is oriented: they are all oriented. So we get $\Sigma_{5,12}$ for this case.
\begin{lstlisting}
f := FreeGroup( "x","y","z");
G := f / [f.1^5, f.2^2, f.3^2, (f.1*f.3)^3, (f.1*f.2)^2,
    (f.2*f.3^-1)^2];
Print(Size(G),"\n"); #120
x:=G.1; y:=G.2; z:=G.3; f2 := FreeGroup( "t");
Z2 := f2 / [f2.1*f2.1]; iso1:=IsomorphismPermGroup(G);
iso2:=IsomorphismPermGroup(Z2); g1:=Image(iso1);
z2:=Image(iso2); img1:=Image(iso1,y); img2:=Image(iso1,x*z);
img3:=Image(iso2,Z2.1); img4:=Image(iso2,Z2.1*Z2.1^-1);
l:=[]; oriented:=0; algebGenus := 1 + Size(G)/6;
if GroupHomomorphismByImages(g1,z2,[img1,img2],[img3,img4])<>fail then
	Print("oriented!\n"); oriented:=1;
else
	Print("non-oriented!\n");
fi;
for e in G do
	conjEle:=e^-1*x*z*e;
	subG:=GroupWithGenerators([y,e^-1*x*z*e]);
	if Index(G, subG)=1 then
		borderSubG := GroupWithGenerators([y*e^-1*x*z*e]);
		b:=Index(G, borderSubG);
        g:=(algebGenus+1-Index(G, borderSubG))/2;
		if ([g,b,1] in l)<>true then
			Add(l, [g,b,1]);
		fi;
		borderSubG := GroupWithGenerators([y^-1*e^-1*x*z*e]);
		b:=Index(G, borderSubG);
		g:=(algebGenus+1-Index(G, borderSubG))/2;
		if ([g,b,1] in l)<>true then
			Add(l, [g,b,1]);
		fi;
		borderSubG := GroupWithGenerators([y, conjEle^-1*y*conjEle]);
		b:=Index(G, borderSubG);
		if oriented=1 then
			g:=(algebGenus+1-Index(G, borderSubG))/2;
		else
			g:=(algebGenus+1-Index(G, borderSubG));
		fi;
		if ([g,b,oriented] in l)<>true then
			Add(l, [g,b,oriented]);
		fi;
		conjEle:=e^-1*(x*z)^-1*e;
		borderSubG := GroupWithGenerators([y, conjEle^-1*y*conjEle]);
		b:=Index(G, borderSubG);
		if oriented=1 then
			g:=(algebGenus+1-Index(G, borderSubG))/2;
		else
			g:=(algebGenus+1-Index(G, borderSubG));
		fi;
		if ([g,b,oriented] in l)<>true then
			Add(l, [g,b,oriented]);
		fi;
	fi;
od;
Print(l,"\n");
\end{lstlisting}
\end{example}

\noindent Chao Wang, School of Mathematical Sciences, University of
Science and Technology of China, 230026 Hefei, CHINA\\
{\it E-mail address:} chao\_{}wang\_{}1987@126.com

\noindent Shicheng Wang, School of Mathematical  Sciences,
100871 Beijing, CHINA\\
{\it E-mail address:} wangsc@math.pku.edu.cn

\noindent Yimu Zhang, Mathematics School, Jilin University,
 130012 Changchun, CHINA\\
{\it E-mail address:} zym534685421@126.com

\noindent Bruno Zimmermann, Dipartimento di Matematica e Geoscienze, Universit\`a degli Studi di Trieste, 34127
Trieste, ITALY\\
{\it E-mail address:} zimmer@units.it
\end{document}